\documentclass{amsart}

\usepackage[hiresbb]{graphicx}
\usepackage{amsmath,amssymb}
\usepackage{amsthm}
\usepackage{amsfonts}
\usepackage{amscd}
\newtheorem{df}{Definition}[section]
\newtheorem{thm}[df]{Theorem}
\newtheorem{prop}[df]{Proposition}
\newtheorem{lemm}[df]{Lemma}
\newtheorem{cor}[df]{Corollary}

\newtheorem{rem}[df]{Remark}

\newcommand{\id}{\mathrm{id}}
\newcommand{\Q}{\mathbb{Q}}

\newcommand{\Z}{\mathbb{Z}}

\newcommand{\shuugou}[1]{\{ #1 \}}
\newcommand{\zettaiti}[1]{\lvert #1 \rvert}

\newcommand{\im}{\mathrm{im}}

\newcommand{\gyaku}[1]{ #1^{-1}}

\newcommand{\skein}[1]{\mathcal{S}( #1 )}

\newcommand{\kukakko}[1]{\langle #1 \rangle}
\newcommand{\defeq}{\stackrel{\mathrm{def.}}{=}}
\newcommand{\Aut}{\mathrm{Aut}}
\newcommand{\bch}{\mathrm{bch}}

\newcommand{\filtn}[1]{\{ #1 \}_{n \geq 0}}

\newcommand{\comp}[1]{\underleftarrow{\lim}_{#1 \rightarrow \infty}}

\newcommand{\arccosh }{\mathrm{arccosh}}

\begin{document}

\title[The Torelli group and skein modules]
{The Torelli group and  the Kauffman bracket skein module }
\author{Shunsuke Tsuji}
\date{}
\maketitle

\begin{abstract}
We introduce an embedding of the Torelli group of a compact connected oriented surface
with non-empty connected boundary into
the completed Kauffman bracket skein algebra of the surface,
which gives a new construction of the first Johnson homomorphism.
\end{abstract}

\section{Introduction}
In \cite{Turaev}, Turaev draw an analogy between the Goldman Lie algebra 
and some skein algebra.
On the other hand, in  
\cite{Kawazumi} \cite{KK} \cite{MT}, Kawazumi, Kuno, Massuyeau and Turaev found that 
the Goldman Lie algebra on a compact connected oriented surface  $\Sigma$ plays 
an important role in study of the mapping class group of 
the surface.
In our preceding paper \cite{TsujiCSAI}, we establish an explicit relationship between the Kauffman
bracket skein algebra and the mapping class group on the surface.
In particular, we obtain a formula for the action of the Dehn twist
along a simple closed curve $c$ on the completed Kauffman bracket skein module
with base point set  an finite set $J \in \partial \Sigma$,
in terms of the inverse function of the hyperbolic cosine function
\begin{equation*}
\exp(\sigma((L(c)))(\cdot)
=t_c (\cdot): \widehat{\skein{\Sigma,J}} \to \widehat{\skein{\Sigma,J}}
\end{equation*}
\cite{TsujiCSAI} Theorem 4.1 where
\begin{equation*}
L(c) \defeq \frac{-A+\gyaku{A}}{4 \log (-A)} (\arccosh (-\frac{c}{2}))^2
-(-A+\gyaku{A})\log (-A).
\end{equation*}
This formula is an analogy of the formula for the action of
the Dehn twist along a simple closed curve $c$ on the completion
of the group ring of the fundamental group of the surface
\cite{Kawazumi} \cite{KK} \cite{MT}.
The aim of this paper is to clarify the relationship between mapping class groups
and Kauffman bracket skein algebras based on the Dehn twists formula
stated above.

Let $\Sigma$ be a compact connected oriented surface with non-empty
connected boundary and genus $g>1$, and
$\widehat{\skein{\Sigma}}$ the completion of the Kauffman bracket skein 
algebra on $\Sigma$ in an $\ker \epsilon$-adic topology 
where $\ker \epsilon$ is an augmentation ideal.
We define a filtration $\filtn{F^n \widehat{\skein{\Sigma}}}$ in
\cite{Tsujipurebraid} satisfying
\begin{equation*}
[F^3 \widehat{\skein{\Sigma}}, F^n \widehat{\skein{\Sigma}}]
\subset F^{n+1} \widehat{\skein{\Sigma}}
\end{equation*}
where Lie bracket $[ \ \ , \ \ ]$ is defined by
\begin{equation*}
[x,y] \defeq \frac{1}{-A+\gyaku{A}} (xy-yx)
\end{equation*}
for $x$ and $y \in \widehat{\skein{\Sigma}}$.
We remark this filtration also induces the topology
induced by the augmentation ideal.
By this filtration, we can consider the logarithm of any element of
the Torelli group of $\Sigma$.

By the above condition, we can define the Baker-Campbell-Hausdorff series 
$\bch (\cdot, \cdot)$ on 
$F^3 \widehat{\skein{\Sigma}}$.
We define $C(c_1,c_2)$ by 
$\bch (L(c_1),L(c_2),-L(c_1),-L(c_2))$ for a pair 
$\shuugou{c_1,c_2}$ of simple closed
curves whose algebraic intersection is $0$.
In view of Putman's result \cite{Pu2008},
we study the three subsets $\mathcal{L}_{\mathrm{comm}}(\Sigma)$,
$\mathcal{L}_{\mathrm{bp}}(\Sigma)$ and
$\mathcal{L}_{\mathrm{sep}}(\Sigma)$
of  $F^3 \widehat{\skein{\Sigma}}$
where
we denote by $\mathcal{L}_{\mathrm{comm}}(\Sigma)$,
$\mathcal{L}_{\mathrm{bp}}(\Sigma)$ and
$\mathcal{L}_{\mathrm{sep}}(\Sigma)$ 
the set of all $C(c_1,c_2)$ for 
a pair $\shuugou{c_1,c_2}$ of two simple closed curves
whose intersection number is $0$,
the set of all $L(c_1)-L(c_2)$ 
for  a bounding pair $\shuugou{c_1,c_2}$
and the set of all $L(c)$
for a separating simple closed curve $c$, respectively.
The subsets $\mathcal{L}_{\mathrm{comm}}(\Sigma)$,
$\mathcal{L}_{\mathrm{bp}}(\Sigma)$ and
$\mathcal{L}_{\mathrm{sep}}(\Sigma)$
correspond to the subsets
$\mathcal{I}_{\mathrm{comm}} (\Sigma)$,
$\mathcal{I}_{\mathrm{bp}}(\Sigma)$
and $\mathcal{I}_{\mathrm{sep}}(\Sigma)$
of the Torelli group of $\mathcal{I} (\Sigma)$
where
we denote by $\mathcal{I}_{\mathrm{comm}}(\Sigma)$,
$\mathcal{I}_{\mathrm{bp}}(\Sigma)$ and
$\mathcal{I}_{\mathrm{sep}}(\Sigma)$ 
the set of all $C_{c_1,c_2}=t_{c_1}t_{c_2}{t_{c_1}}^{-1}
{t_{c_2}}^{-1}$ for 
a pair $\shuugou{c_1,c_2}$ of two simple closed curves
whose intersection number is $0$,
the set of all $t_{c_1}{t_{c_2}}^{-1}$ 
for  a bounding pair $\shuugou{c_1,c_2}$
and the set of all $t_c$
for a separating simple closed curve $c$, respectively.
In fact, in subsection \ref{subsection_injectivity_of_zeta},
we construct a surjective homomorphism
$\theta : (I \skein{\Sigma}, \bch) \to \mathcal{I}(\Sigma)$
defined by
$C(c_1,c_2) \in \mathcal{L}_{\mathrm{comm}}(\Sigma) \mapsto C_{c_1,c_2}
\in \mathrm{I}_{\mathrm{comm}}(\Sigma)$,
$L(c_1)-L(c_2) \in \mathcal{L}_{\mathrm{bp}}(\Sigma) \mapsto 
t_{c_1}{t_{c_2}}^{-1} \in \mathcal{I}_{\mathrm{bp}}(\Sigma)$
and $L(c) \in \mathcal{L}_{\mathrm{sep}}(\Sigma)
\mapsto t_{c} \in \mathcal{I}_{\mathrm{sep}}(\Sigma$
where we denote by $I\skein{\Sigma}$
the subgroup of $(F^3 \widehat{\skein{\Sigma}},\bch)$
generated by 
$\mathcal{L}_{\mathrm{comm}}(\Sigma)$,
$\mathcal{L}_{\mathrm{bp}}(\Sigma)$ and
$\mathcal{L}_{\mathrm{sep}}(\Sigma)$.
This homomorphism satisfies
\begin{equation*}
\exp  (\sigma (x)) (\cdot) =\theta (x) (\cdot) \in \Aut (\widehat{\skein{\Sigma,J}})
\end{equation*}
for any $x \in I \skein{\Sigma}$ and any finite subset $J$ of 
$\partial \Sigma$. Using the Putman's infinite presentation 
\cite{Pu2008}, we prove that $\theta$ is injective
in subsection \ref{subsection_well_defined_zeta}.

This embedding gives us a new way of
studying the mapping class group.
In fact, the embedding  gives 
a new construction of the first Johnson homomorphism
by
\begin{equation*}
\mathcal{I}(\Sigma) \to I \skein{\Sigma} \hookrightarrow
F^3 \widehat{\skein{\Sigma}} \twoheadrightarrow
F^3 \widehat{\skein{\Sigma}}/F^4 \widehat{\skein{\Sigma}}
\stackrel{\lambda^{-1}}{\simeq} 
\wedge^3 H_1 (\Sigma,\Q).
\end{equation*}
where $\zeta \defeq \theta^{-1}$.
Here $\wedge^n H_1 (\Sigma,\Q)$ the 
$n$-th exterior power of $H_1 (\Sigma,\Q)$.
These are analogies of  \cite{KK} 6.3.
In a subsequent paper,
we study the Johnson kernel $\mathcal{K} (\Sigma)$,
which is defined to be the kernel of the first Johnson homomorphism.
The isomorphism 
$\zeta_{|\mathcal{K} (\Sigma)} :\mathcal{K} (\Sigma) \to
F^4 \widehat{\skein{\Sigma}} \cap  I  \skein{\Sigma}$
induces the natural map
$\tau_{\zeta 2}: \mathcal{K} (\Sigma) \to
F^4 \widehat{\skein{\Sigma}}/F^5 \widehat{\skein{\Sigma}}
\simeq S^2 (S^2 (H_1 (\Sigma,\Q))) \oplus S^2 (H_1 (\Sigma, \Q)) \oplus
\Q$ where we denote by
$S^2 (V)$  the second symmetric tensor of 
a $\Q$ vector space $V$.
Then we have $\tau_{\zeta 2} =(\iota \circ \tau_2) \oplus 0 \oplus d_{\mathrm{Casson}}$.
Here $\iota: S^2 (\wedge^2 (H_1 (\Sigma, \Q)))
\to S^2 (S^2 (H_1(\Sigma,\Q)))$ is the $\Q$-linear map
defined by $(a \wedge b) \cdot (c \wedge d) \to (a \cdot c) \cdot(b\cdot d)-
(a \cdot d) \cdot (b \cdot c)$, $\tau_2 $ the second Johnson homomorpshim
and, $d_{\mathrm{Casson}}$  the core of the Casson invariant
defined in \cite{Morita_Casson_core}.

Furthermore, we expect that the embedding
 brings us some information about integral homology $3$-spheres
including the Casson invariant.
In a subsequent paper \cite{Tsujihom3},
we construct an invariant $z(M)$ for an integral homology $3$-sphere
$M$
which is an element of $\Q [[A+1]]$
using the embedding $\zeta$.
We remark the coefficient of $(A+1)$ in $z(M)$
is the Casson invariant. 

In section 2, we review some facts about
the skein algebra and the mapping class group of 
a compact connected surface.
In section 3, we construct 
an embedding from $\mathcal{I} (\Sigma) \to  F^3 \widehat{\skein{\Sigma}}$.

\begin{rem}
In section 3, we assume that
the genus of  a surface $\Sigma$ is larger than $2$
in order to use Putman's theorem.
Since $\mathcal{I} (\Sigma_{1,1})$ is generated by only
the Dehn twist along the simple slosed curve
which is parallel to the boundary of $\Sigma_{1,1,}$,
the map $\zeta$ is aloso isomorphism.

\end{rem}

\section*{Acknowledgment}
The author would like to thank his adviser, Nariya Kawazumi, for helpful discussion
and encouragement. 
This work was supported by JSPS KAKENHI Grant Number 15J05288 
and the Leading Graduate Course for Frontiers of Mathematical Sciences and Phsyics.

\tableofcontents
\section{Definition and Review}
\label{section_definition_review} 
In the section, we review some definitions and facts about the filtered 
Kauffman bracket skein algebra and module of a surface,
for details, see our papers \cite{TsujiCSAI} and \cite{Tsujipurebraid}.

Through this section, let $\Sigma$ be a compact connected surface
with non-empty boundary
and $I$ the closed interval $[0,1]$.

\subsection{Kauffman bracket skein algebras and modules}
Let $J$ be a finite subset of $\partial \Sigma$.
We denote by $\mathcal{T}(\Sigma,J)$ the set of unoriented framed tangles in
$\Sigma \times I$ with base point set $J$ 
and by $T(d)$ the tangle presented by a tangle diagram $d$,
 for details, see \cite{TsujiCSAI} section 2 and \cite{Tsujipurebraid} section 2.
Let $\skein{\Sigma,J}$ be the Kauffman bracket skein module of 
$\Sigma$ with base point set $J \times \shuugou{\frac{1}{2}}$, 
which is the quotient of $\Q [A.\gyaku{A}] \mathcal{T}(\Sigma,J)$ 
by the skein relation and the trivial knot relation \cite{TsujiCSAI} Definition 3.2.
 The skein relation is 
\begin{equation*}
T(d_1) -AT(d_\infty)-\gyaku{A} T(d_0)
\end{equation*}
where $d_1$, $d_\infty$ and $d_0$ are differ
only in an open disk shown in Figure 
\ref{fig_K2},
Figure \ref{fig_Kinfi} and Figure \ref{fig_K0}, respectively.
The trivial knot relation is
\begin{equation*}
T(d) -(-A^2-A^{-2})T(d')
\end{equation*}
where $d$ and $d'$ are differ only in a open disk shown 
a boundary of a disk and empty, respectively.
We denote that  we don't assume  "the boundary skein relation"
and "the value of  a contractible arc" in Muller \cite{Mu2012}.
We write simply $\skein{\Sigma} \defeq \skein{\Sigma, \emptyset}$.
The element of $\skein{\Sigma,J}$ represented by $T \in \mathcal{T}(\Sigma,J)$
is denoted by  $[T]$. 
%The equation
%\begin{align}
%\label{equation_kouten_sa}
%&[T(d_1)]-[T(d_2)] =(A-\gyaku{A})([T(d_\infty)]-[T(d_0)]),
%\end{align}
%is very useful,
%where 
%$d_1$,$d_2$, $d_\infty$ and $d_0$ are differ only 
%in an open disk shown in Figure 
%\ref{fig_K2}, Figure \ref{fig_K1},
%Figure \ref{fig_Kinfi} and Figure \ref{fig_K0}, respectively.

\begin{figure}[htbp]
	\begin{tabular}{rrrr}
	\begin{minipage}{0.25\hsize}
		\centering
		\includegraphics[width=2cm]{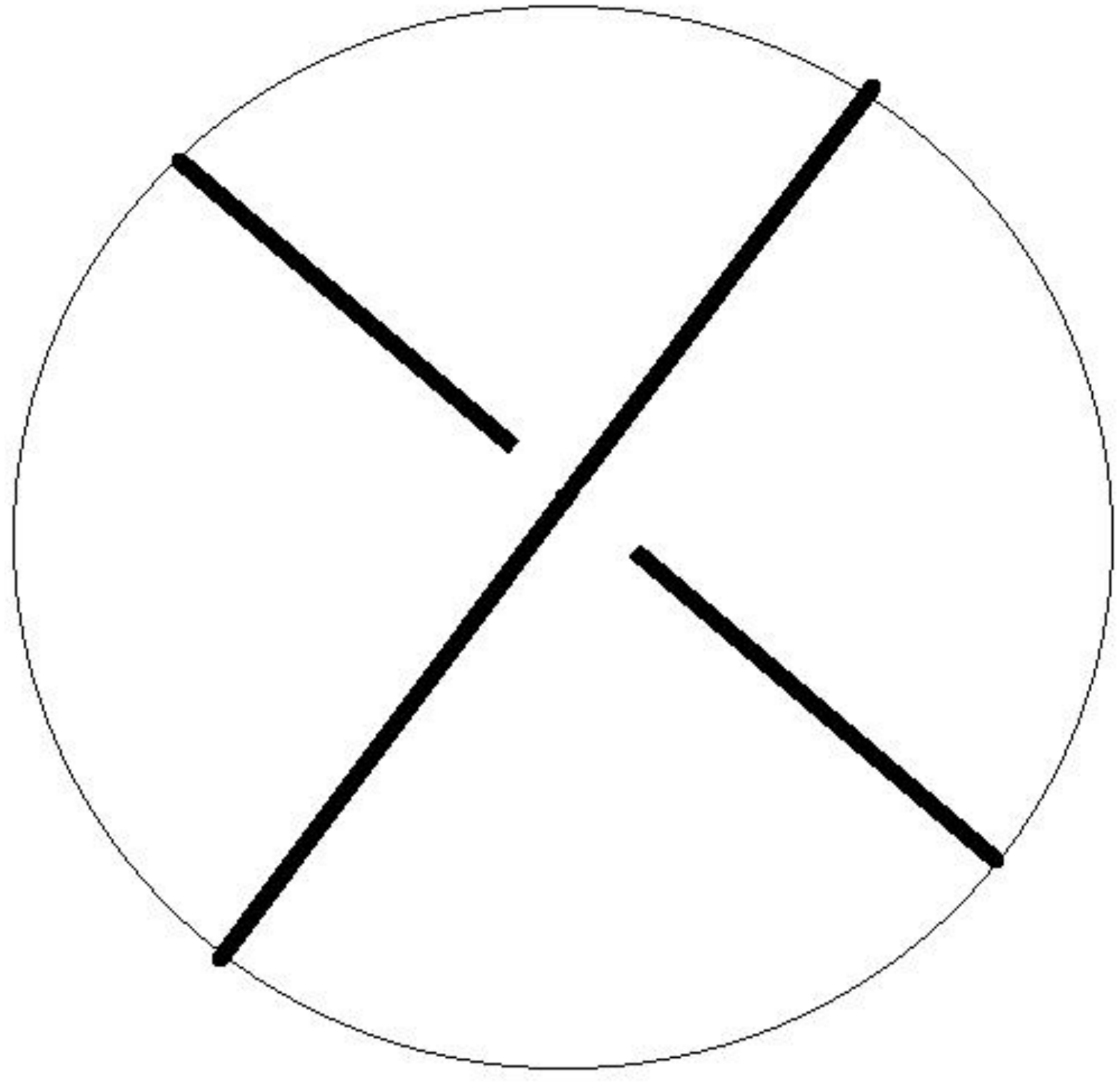}
		\caption{$d_1$}
		\label{fig_K2}
	\end{minipage}
     %\begin{minipage}{0.25\hsize}
	%	\centering
	%	\includegraphics[width=2cm]{K1_PNG.pdf}
	%	\caption{$d_2$}
	%	\label{fig_K1}
	%\end{minipage}
		\begin{minipage}{0.25\hsize}
		\centering
		\includegraphics[width=2cm]{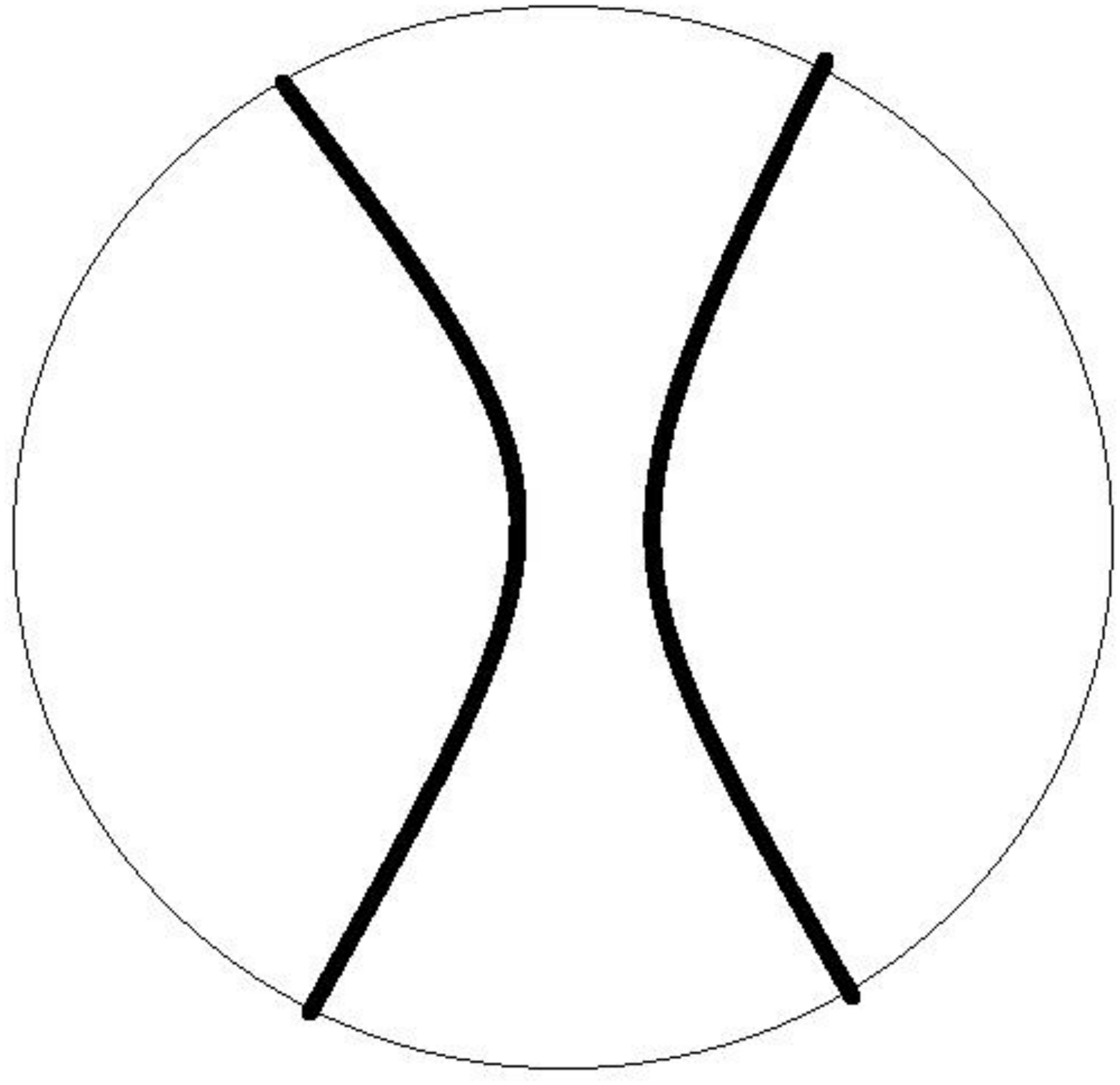}
		\caption{$d_\infty$}
		\label{fig_Kinfi}
	\end{minipage}
		\begin{minipage}{0.25\hsize}
		\centering
		\includegraphics[width=2cm]{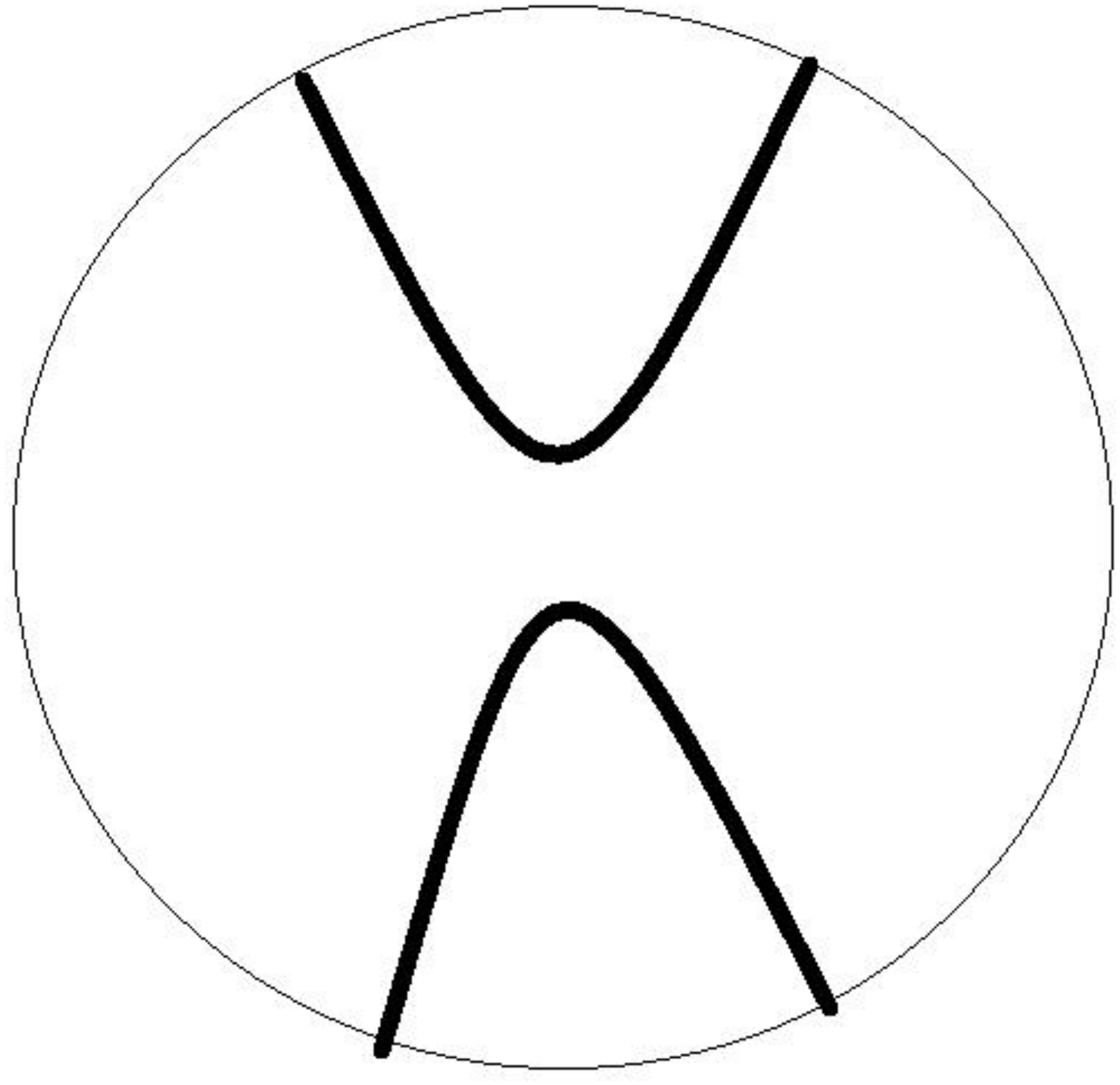}
		\caption{$d_0$}
		\label{fig_K0}
	\end{minipage}
	\end{tabular}
\end{figure}

%We denote by $\mathcal{M}(\Sigma)$ the mapping class group of $\Sigma$
%fixing the boundary pointwise.
There is a natural action of 
$\mathcal{M}(\Sigma)$ on $\skein{\Sigma,J}$ \cite{TsujiCSAI} section 2,
where $\mathcal{M}(\Sigma)$ is the mapping class group of $\Sigma$
fixing the boundary pointwise.
The product of $\skein{\Sigma}$ and
the right and left actions of $\skein{\Sigma}$ on $\skein{\Sigma,J}$
are defined by Figure \ref{fig_product_action}, for details, see \cite{TsujiCSAI} 3.1.
The Lie bracket $[,]:\skein{\Sigma} \times \skein{\Sigma} \to \skein{\Sigma}$ is defined by
$[x,y]\defeq \frac{1}{-A+\gyaku{A}} (xy-yx)$.
and the action $\sigma()(): \skein{\Sigma} \times \skein{\Sigma,J}
\to \skein{\Sigma,J}$ by
$\sigma(x)(z) \defeq \frac{1}{-A+\gyaku{A}}(xz-zx)$.
The action $\sigma$
makes $\skein{\Sigma,J} $ a $(\skein{\Sigma},[,])$-module with $\sigma$.
For details, see,  \cite{TsujiCSAI} 3.2.

\begin{figure}
\begin{picture}(300,83)
\put(0,-20){\includegraphics[width=90pt]{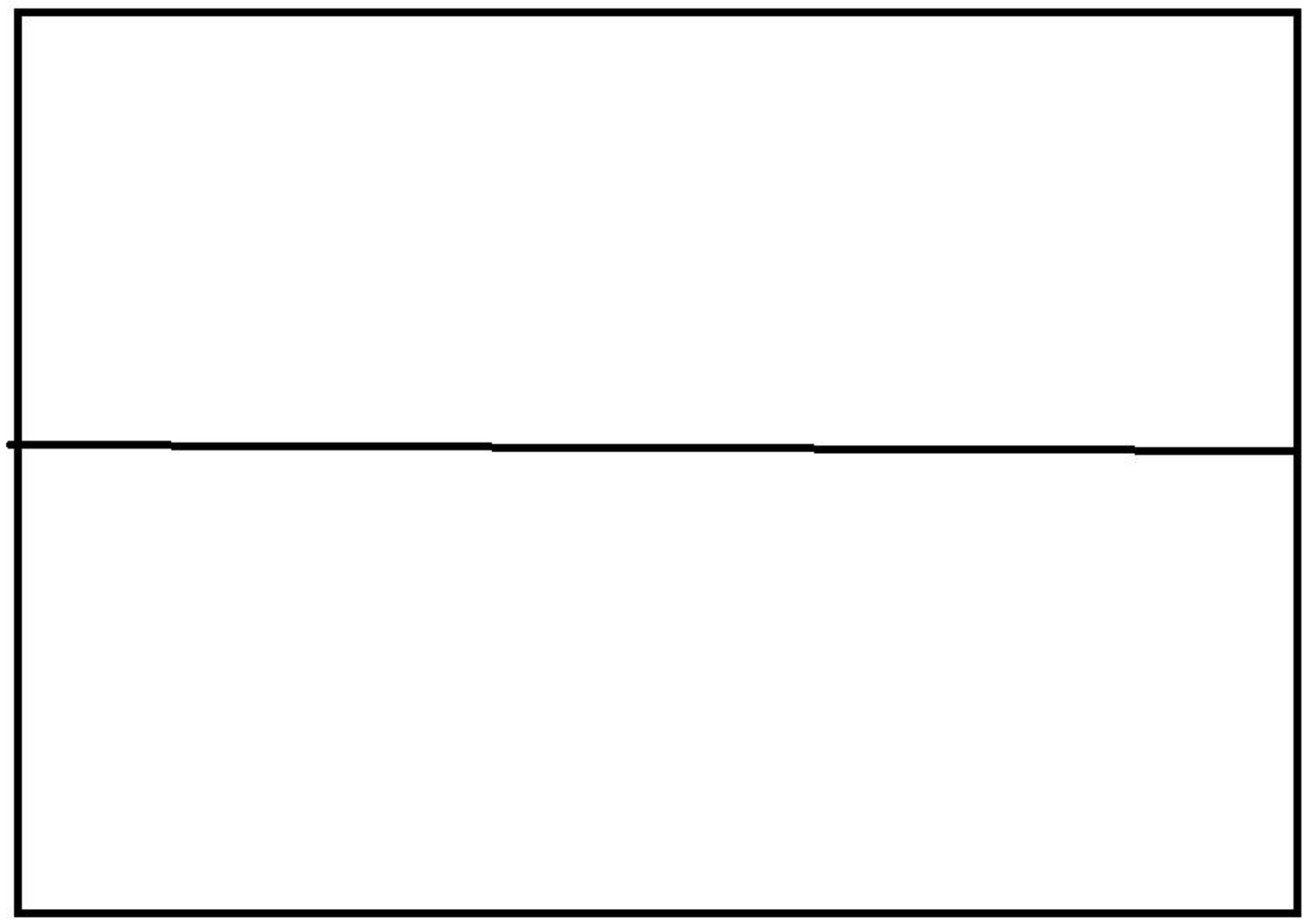}}
\put(100,-20){\includegraphics[width=90pt]{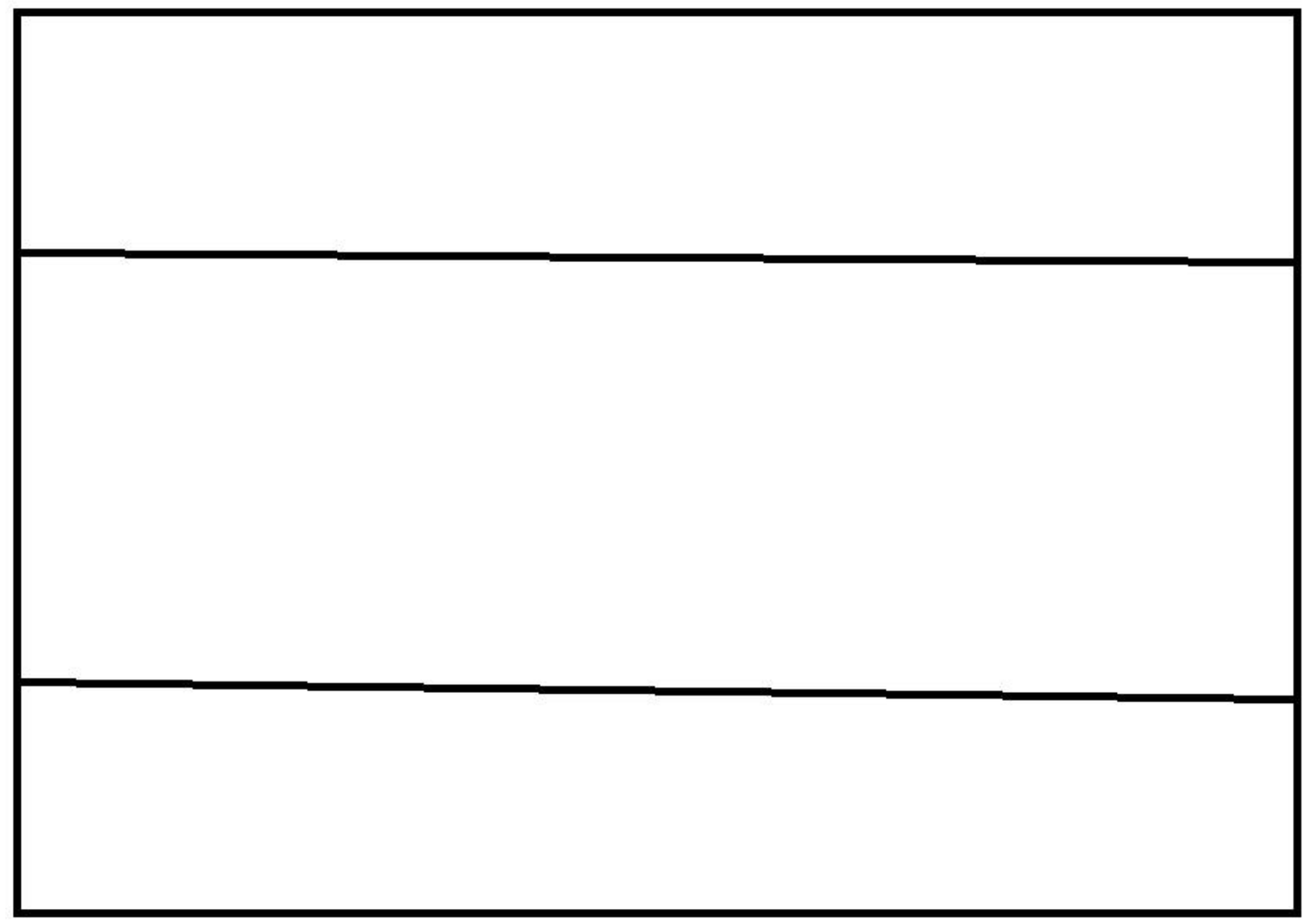}}
\put(200,-20){\includegraphics[width=90pt]{sikaku_3_PNG.pdf}}
\put(10,80){$xy \defeq$}
\put(40,50){$x$}
\put(40,20){$y$}
\put(10,2){$\mathrm{for} \ \  x,y \in \skein{\Sigma}$}
\put(0,10){$0$}
\put(0,55){$1$}
\put(0,32.5){$I$}
\put(40,65){$\Sigma$}

\put(110,80){$xz \defeq$}
\put(140,52){$x$}
\put(140,34){$z$}
\put(110,2){$\mathrm{for} \ \  x\in \skein{\Sigma}$}
\put(110,-10){$\ \ \mathrm{and} \ \ z \in \skein{\Sigma,J}$}
\put(100,10){$0$}
\put(100,55){$1$}
\put(100,32.5){$I$}
\put(140,65){$\Sigma$}

\put(210,80){$zx \defeq$}
\put(240,34){$z$}
\put(240,15){$x$}
\put(210,2){$\mathrm{for} \ \  x \in \skein{\Sigma}$}
\put(210,-10){$\ \ \mathrm{and} \ \ z \in \skein{\Sigma,J}$}
\put(200,10){$0$}
\put(200,55){$1$}
\put(200,32.5){$I$}
\put(240,65){$\Sigma$}

\end{picture}
\caption{The product and the actions}
\label{fig_product_action}
\end{figure}

\subsection{Filtration and Lie bracket}
The augmentation map $ \epsilon:\skein{\Sigma} \to \Q$ is defined by
$A+1 \mapsto 0$ and $[L]- (-2)^{\zettaiti{L}}\mapsto 0$ for $L \in \mathcal{T}(\Sigma)$,
where $\zettaiti{L}$ is the number of components of $L$.
The well-definedness of $\epsilon $ is proved in \cite{TsujiCSAI} Proposition
3.10. 

We denote $\pi \defeq \pi_1 (\Sigma, *)$ and
by $\Q \pi$ the group ring of $\pi$ over $\Q$,
where $*$ is a point of $\partial \Sigma$.
For $x \in \pi_1(\Sigma)$,
we define $\kukakko{x} \in (\ker \epsilon)/(\ker \epsilon)^2 $ 
by $[L_x]+2-3 w(L_x) (A-\gyaku{A})$ using $L_x \in \mathcal{T}(\Sigma)$
with $p_1 (L_x)$ the conjugation class of $x$, where
the writhe $w(L_x)$ is the sum of the signs of the crossing of a diagram 
presenting $L_x$.
The well-definedness of $\kukakko{\cdot}$ is proved
in \cite{Tsujipurebraid} Lemma 3.1.
We also denote by $\kukakko{\cdot} :\Q \pi
\ker \epsilon / ( \ker \epsilon)^2$
its $\Q$-linear extension.
By \cite{Tsujipurebraid} Proposition 3.3. we have
$\kukakko{xy}+\kukakko{x\gyaku{y}} =2 \kukakko{x}+2 \kukakko{y}$
for $x$ and $y \in \pi$.
Using this formula, we have
\begin{align}
&\kukakko{(a-1)(b-1)(c-1)} = \kukakko{-(b-1)(a-1)(c-1)}, \label{equation_calc_1}\\
&\kukakko{(a-1)(b-1)(c-1)(d-1)} =0, \label{equation_calc_2}\\
&\kukakko{(a-1)(b-1)(b-1)} = 0, \label{equation_calc_3}\\
&\kukakko{[a,b]c-c} = 2\kukakko{(a-1)(b-1)(c-1)}, \label{equation_calc_4}\\
&\kukakko{(a-1)(a-1)}=2\kukakko{a}, \label{equation_calc_5}
\end{align}
for $a$, $b$, $c$ and $d \in \pi$.
For details, see \cite{Tsujipurebraid} Lemma 3.4.
By these equations, the $\Q$-linear map
\begin{equation*}
\lambda :H \wedge H \wedge H
\to \ker \epsilon / (\ker \epsilon)^2, [a]\wedge [b] \wedge [c]
\mapsto \kukakko{(a-1)(b-1)(c-1)}
\end{equation*}
is well-defined where $H \defeq H_1 (\Sigma ,\Q) =\Q \otimes
\pi / [\pi , \pi]$.
We remark that $\lambda$ is injective Corollary 4.6.
Let $\varpi$ be the quotient map 
$\ker \epsilon \to \ker \epsilon /\im \lambda$. 
We define the filtration $\filtn{F^n \skein{\Sigma}}$
by
\begin{align*}
&F^0  \defeq \skein{\Sigma}, \\
&F^1 \skein{\Sigma} = F^2 \skein{\Sigma} \defeq \ker \epsilon, \\
&F^3 \skein{\Sigma} \defeq  \ker \varpi, \\
&F^n \skein{\Sigma} \defeq \ker \epsilon F^{n-2} \skein{\Sigma} \ \ (\mathrm{for} \ \ 4 \leq n).
\end{align*}
By \cite{Tsujipurebraid} Proposition 5.7. and Proposition 5.11,
we have
\begin{align*}
&F^m \skein{\Sigma} F^n \skein{\Sigma} \subset F^{n+m} \skein{\Sigma}, \\
&[F^m \skein{\Sigma}, F^n \skein{\Sigma}] \subset F^{n+m-2} \skein{\Sigma}, 
\end{align*}
for any $n$ and $m$.
Let $\rho$  be the $\Q$-linear map
\begin{align*}
\rho :H \bullet H \to F^2 \skein{\Sigma}/ F^3 \skein{\Sigma}, \ \ [a] \bullet [b] \mapsto
\kukakko{(a-1)(b-1)},
\end{align*}
where $H \bullet H$ is the symmetric tensor of $H$.
We remark $\rho$ is a $\Q$-module incjection by \cite{Tsujipurebraid} Theorem 4.1.

\subsection{Completion and Dehn twist}
We denote 
\begin{align*}
&\widehat{\skein{\Sigma}} \defeq \comp{i} \skein{\Sigma}/(\ker \epsilon)^i, \\
&\widehat{\skein{\Sigma,J}} \defeq \comp{i} \skein{\Sigma,J}/
(\ker \epsilon)^i \skein{\Sigma,J},
\end{align*}
for a finite subset $J \subset \partial \Sigma$.
By \cite{TsujiCSAI} Theorem 5.5, the natural homomorphisms
$\skein{\Sigma} \to \widehat{\skein{\Sigma}} $ and
$\skein{\Sigma,J} \to \widehat{\skein{\Sigma,J}}$
is injective.

\begin{prop}[\cite{Tsujipurebraid} Corollary 2.4.]
\label{prop_map_inj}
Let $J$ be a finite subset of $\partial \Sigma$.
If $\Sigma \neq \emptyset$ and $\partial_i \cap J \neq \emptyset$
for each $i \in \shuugou{1,2, \cdots,b}$, then
$\mathcal{M}(\Sigma) \to \Aut (\widehat{\skein{\Sigma,J}})$ is injective,
where $\partial_1, \cdots, \partial_b$ are the connected components of $
\partial \Sigma$.
\end{prop}

We denote 
\begin{equation}
\label{equation_L_c}
L(c) \defeq \frac{-A+\gyaku{A}}{4 \log(-A)} (\arccosh (-\frac{c}{2}))^2
-(-A+\gyaku{A})\log (-A)
\end{equation}
where $c$ is also denoted by the element of $\skein{\Sigma}$ represented by 
the knot presented by a simple closed curve $c$.

\begin{thm}[\cite{TsujiCSAI} Theorem 4.1]
\label{thm_Dehn_twist}
Let $c$  be a simple
closed curve and $t_c$ the Dehn twist along $c$.
Then we have
\begin{equation*}
t_c(\cdot) = \exp (\sigma(L(c))) \defeq \sum_{i=0}^\infty \frac{1}{i!} (\sigma(L(c)))^i
\in \Aut (\widehat{\skein{\Sigma,J}}).
\end{equation*}
for any finite subset $J \subset \partial \Sigma$.
\end{thm}

\subsection{The Baker-Campbell-Hausdorff series}
\label{subsection_bch}
In this subsection, we will explain the Baker-Campbell-Hausdorff series.
We choose $S \subset \widehat{\skein{\Sigma}}$ such that,
for any $i \in \Z_{\geq 0}$, there exists $j_i\in \Z_{\geq 0}$
satisfying 
\begin{equation}
\label{equation_jouken_bch}
\sigma(a_{1})\circ
\sigma(a_{2}) \circ \cdots \sigma ( a_{j_i})(F^i
\skein{\Sigma})\subset F^{i+1} \skein{\Sigma}
\end{equation}
for $a_1, a_2, \cdots, a_{j_i} \in
S$.
In this paper,  the Baker-Campbell-Hausdorff series $\bch$ is defined by
\begin{align*}
&\bch (\epsilon_1 a_1, \epsilon_2 a_2, \cdots,\epsilon_m a_m) \\
&\defeq 
(-A+\gyaku{A}) \log (\exp(\frac{\epsilon_1 a_1}{-A+\gyaku{A}})
\exp(\frac{\epsilon_2 a_2}{-A+\gyaku{A}}) \cdots
\exp(\frac{\epsilon_m a_m}{-A+\gyaku{A}}))
\end{align*}
for $a_1, a_2, \cdots, a_{m} \in
S$ and $\epsilon_1, \epsilon_2, \cdots, \epsilon_m$.
We remark, as elements of the associated Lie algebra
$(\widehat{\skein{\Sigma}},[ \ \ , \ \ ])$,
it has a usual expression.
For example, 
\begin{equation*}
\bch(x,y) = x+y+\frac{1}{2}[x,y]+\frac{1}{12}([x,[x,y]]+[y,[y,x]])+ \cdots.
\end{equation*}
We denote
\begin{equation*}
\zettaiti{S} \defeq
\shuugou{\bch (\epsilon_1 a_1, \cdots, \epsilon_m 
a_m)|m \in \Z_{\geq 1}, a_1,  \cdots, a_{m} 
\in S, \epsilon_1, \cdots, \epsilon_m \shuugou{\pm1}}.
\end{equation*}
For $a_1, a_2, \cdots, a_{j_1}  \in \zettaiti{S}$, 
they satisfies the equation (\ref{equation_jouken_bch}).
The Baker-Campbell-Hausdorff series satisfies
\begin{align*}
& \bch(a,-a) =0, \\
& \bch(0,a)=\bch(a,0) =a, \\
& \bch(a,\bch(b,c))=\bch(\bch(a,b),c), \\
& \bch(a, b, -a) =\exp(\sigma(a))(b).
\end{align*}
Hence $(\zettaiti{S}, \bch) $ is a group whose identity is $0$.
Furthermore, if there exists $j' \in \Z_{>0}$ satisfying
\begin{equation}
\label{equation_bch_jouken_aut}
\sigma(a_{1})\circ
\sigma(a_{2}) \circ \cdots \sigma (a_{j'})(
\skein{\Sigma,J})\subset \ker{\epsilon} \skein{\Sigma,J}
\end{equation}
for $a_1, a_2, \cdots, a_{j'} \in
S$ and any finite set $J \subset \partial{\Sigma}$,
$\exp :(\zettaiti{S},\bch) \to \Aut (\widehat{\skein{\Sigma,J}})$
is a group homomorphism, i.e, $\exp(\sigma(\bch(a,b))) =\exp(\sigma(a)) \circ \exp
(\sigma(b))$ for $a,b \in \zettaiti{S}$.
For example, $F^3 \widehat{\skein{\Sigma}}$ satisfies the condition.
Furthermore, if the genus of $\Sigma$ is $0$,
$\widehat{\skein{\Sigma}}$ satisfies the condition and
there exists an embedding $\mathcal{M} (\Sigma) \hookrightarrow
(\widehat{\skein{\Sigma}}, \bch).$ For details, see \cite{Tsujipurebraid}.

\begin{lemm}[\cite{Tsujipurebraid} Corollary 5.16. Proposition 5.17.]
\label{lemm_bch_jouken}
Let $V_1$ and $V_2 \subset V_1$ be $\Q$-linear subspaces of $H$
satisfying $\mu (v,v') =0$ for any $v \in V_2 $ and $v' \in V_1$.
We denote
\begin{equation*}
S = \shuugou{x  \in \widehat{\skein{\Sigma}}|
\rho^{-1} (x \mathrm{ \ mod} F^3 \widehat{\skein{\Sigma}}) \in
V_1 \bullet V_2}.
\end{equation*}
Then $S$ satisfies the above conditions (\ref{equation_jouken_bch})
and
(\ref{equation_bch_jouken_aut}).
\end{lemm}

\section{Torelli groups and Johonson homomorphisms}
\label{section_Torelli}
Through this section, let $\Sigma$ be a compact surface
with a non-empty connected boundary and genus $g>1$.

The completed Kauffman bracket $\widehat{\skein{\Sigma}}$
has a filtration $\filtn{F^n \widehat{\skein{\Sigma}}}$
such that $\skein{\Sigma}/F^n \skein{\Sigma}\simeq \widehat{
\skein{\Sigma}}/F^n \widehat{\skein{\Sigma}}$ for
$n \in \Z_{\geq 0}$.

\subsection{The definition of $I \skein{\Sigma}$}
\label{subsection_definition_I}
The aim of this subsection is to define $I \skein{\Sigma}$.

\begin{lemm}
Let $\shuugou{c_1, c_2}$ be a pair of curves whose algebraic intersection number
is $0$. 
Then the set $\shuugou{L(c_1),L(c_2)}$ satisfies
the conditions  (\ref{equation_jouken_bch}) 
and (\ref{equation_bch_jouken_aut}).
\end{lemm}

\begin{proof}
Let $\gamma_1$ and $\gamma_2$ be elements of $H_1$ such that 
$\pm [c_1] =\gamma_1$ and $\pm [c_2] =\gamma_2$.
We remark that
\begin{equation*}
L(c_i) =-\frac{1}{2}(c_i+2) =-\frac{1}{4}\rho(\gamma_i \bullet \gamma_i) \mod F^3 
\widehat{\skein{\Sigma}}
\end{equation*}
for $i =1,2$.
Since $\mu(\gamma_1,\gamma_2)=0$ and Lemma \ref{lemm_bch_jouken},
we have the claim of the Lemma.
This finishes the proof.
\end{proof} 

By this lemma, we can define $C(c_1,c_2) =\bch (L(c_1),L(c_2),-L(c_1),-L(c_2))$
for a pair $\shuugou{c_1, c_2}$ of curves whose algebraic intersection number
is $0$. We denote  by $\mathcal{L}_{\mathrm{comm}}(\Sigma)$
the set of all $C(c_1,c_2) $. Since $[L(c_1),L(c_2)] \in F^3 \widehat{\skein{\Sigma}}$ for
a pair $\shuugou{c_1,c_2}$ of curves whose algebraic intersection number is $0$,
we have $\mathcal{L}_{\mathrm{comm}} (\Sigma) \subset F^3
\widehat{\skein{\Sigma}}$.

\begin{lemm}
\label{lemm_bounding_pair}
Let $\shuugou{c_1,c_2}$ be a pair of non-isotopic disjoint homologous curves
which is presented by $(r[a_1,b_1]\cdots[a_m,b_m])_\square$ and $(r)_\square$, 
respectively,
for $r, a_1, b_1, \cdots, a_m, b_m \in \pi_1(\Sigma)$.
Then we have $L(c_1)-L(c_2) =-\sum_{i=1}^m \lambda ([r]\wedge
[a_i]\wedge [b_i]) \mod F^4 \widehat{
\skein{\Sigma} } = ( \ker \epsilon)^2$.
\end{lemm}

\begin{proof}
By the equation (\ref{equation_calc_4}), we have
\begin{align*}
&L(c_1)-L(c_2)=-\frac{1}{2}(c_1-c_2) =-\frac{1}{2}\kukakko{r \prod_{i=1}^m [a_i,b_i]-r} \\
&=-\frac{1}{2}\sum_{j=1}^m \kukakko{r\prod_{i=1}^j[a_i,b_i]-r\prod_{i=1}^{j-1}[a_i,b_i]} 
=-\sum_{j=1}^m\kukakko{(r\prod_{i=1}^{j-1}[a_i,b_i]-1)(a_j-1)(b_j-1)} \\
&=-\sum_{i=1}^m \lambda ([r]\wedge[a_i]\wedge [b_i]).  \mod (\ker \epsilon)^2\\
\end{align*}
This finishes the proof.
\end{proof}

\begin{cor}
For  a separating simple closed curve $c$,
we have $L(c) =0 \mod (\ker \epsilon)^2.$
\end{cor}

We denote by $\mathcal{L}_{\mathrm{bp}}(\Sigma)$ and
$\mathcal{L}_{\mathrm{sep}}(\Sigma)$
the set of all $L(c_1)-L(c_2)$ for 
 a pair $\shuugou{c_1,c_2}$
 of non-isotopic disjoint homologous curves,
i.e. which is a bounding pair (BP)
and the set of all $L(c)$
for a separating simple closed curve $c$,
i.e. for separating simple closed curve (sep. s.c.c.) $c$, respectively.
By the above lemma and corollary,
we have $\mathcal{L}_{\mathrm{bp}}(\Sigma) \cup
\mathcal{L}_{\mathrm{sep}}(\Sigma) \subset F^3 \widehat{\skein{\Sigma}}$.
We define
\begin{equation*}
I \skein{\Sigma} \defeq \zettaiti{\mathcal{L}_{\mathrm{gen}}(\Sigma)}
\end{equation*}
where we denote $\mathcal{L}_{\mathrm{gen}}(\Sigma) \defeq
\mathcal{L}_{\mathrm{comm}} (\Sigma)
\cup \mathcal{L}_{\mathrm{bp}}(\Sigma) \cup
\mathcal{L}_{\mathrm{sep}}(\Sigma)$.

We denote by
\begin{align*}
&\mathcal{I}_{\mathrm{comm}}(\Sigma) \defeq
\shuugou{C_{c_1c_2} \defeq [t_{c_1}, t_{c_2}]|\mu(c_1,c_2)=0}, \\
&\mathcal{I}_{\mathrm{sep}}(\Sigma) \defeq
\shuugou{t_{c_1c_2} \defeq t_{c_1}{t_{c_2}}^{-1}|\shuugou{c_1,c_2}: \mathrm{BP}}, \\
&\mathcal{I}_{\mathrm{sep}}(\Sigma) \defeq
\shuugou{t_{c}|c: \mathrm{sep. \ \ s.c.c.}}.
\end{align*}
The set $\mathcal{I}_{\mathrm{gen}}(\Sigma) \defeq
\mathcal{I}_{\mathrm{comm}} (\Sigma)
\cup \mathcal{I}_{\mathrm{bp}}(\Sigma) \cup
\mathcal{I}_{\mathrm{sep}}(\Sigma)$ is the generator set
of the infinite presentation of $\mathcal{I} (\Sigma)$ in \cite{Pu2008}.

\begin{df}
The map $\theta :\mathcal{L}_{\mathrm{gen}}(\Sigma)  \to \mathcal{I}_{\mathrm{gen}}(\Sigma)$ is
defined by the following.
\begin{itemize}
\item  For
a pair $\shuugou{c_1,c_2}$ of curves whose algebraic intersection number is $0$,
$C(c_1,c_2) \mapsto C_{c_1c_2}$.
\item For a pair $\shuugou{c_1,c_2}$
 of non-isotopic disjoint homologous curves, $L(c_1)-L(c_2) \mapsto t_{c_1c_2}$.
\item For
a separating simple closed curve $c$,
$\theta (L(c)) =t_c$.
\end{itemize}

\end{df}

In subsection \ref{subsection_injectivity_of_zeta}
Proposition \ref{prop_theta}, we prove $\theta$ is well-defined.

\subsection{Well-definedness of $\theta$}
\label{subsection_injectivity_of_zeta}

The aim of this subsection is to prove the following.

\begin{lemm}
\label{lemm_zeta_inv_inj}
The map $\theta:\mathcal{L}_{\mathrm{gen}}(\Sigma) \to
\mathcal{I}_{\mathrm{gen}}(\Sigma)$ induces
$\theta:I \skein{\Sigma} \to \mathcal{I}(\Sigma)$.
\end{lemm}

\begin{prop}
\label{prop_theta}
The map $\theta$ is well-defined.
\end{prop}

\begin{proof}
By Theorem \ref{thm_Dehn_twist}, we have the followings.

\begin{itemize}
\item  For
a pair $\shuugou{c_1,c_2}$ of curves whose algebraic intersection number is $0$,
$\exp(\sigma(C(c_1,c_2)))(\cdot)= C_{c_1c_2}(\cdot)\in \Aut (\widehat{\skein{\Sigma,J}})$ for any finite set $J \in \partial{\Sigma}$.
\item For a pair $\shuugou{c_1,c_2}$
 of non-isotopic disjoint homologous curves, 
$\exp(\sigma(L(c_1)-L(c_2))) (\cdot)= t_{c_1c_2}(\cdot)\in \Aut (\widehat{\skein{\Sigma,J}})$ for any finite set $J \in \partial{\Sigma}$.
\item For
a separating simple closed curve $c$,
$\exp(\sigma(L(c)))(\cdot) =t_c(\cdot) \in \Aut (\widehat{\skein{\Sigma,J}})$
 for any finite set $J \in \partial{\Sigma}$.
\end{itemize}

By Proposition \ref{prop_map_inj}, $\theta$ is well-defined.
This finishes the proof.
\end{proof}

By Theorem \ref{thm_Dehn_twist}, we have the following.

\begin{lemm}
\label{lemm_theta}
For $x_1,x_2, \cdots, x_j  \in \shuugou{\epsilon x|\epsilon \in \shuugou{\pm1},
x \in \mathcal{L}_{\mathrm{gen}}(\Sigma)}$,
 we have 
\begin{equation*}
\prod_{i=1}^j(\theta ( x_i)) (\cdot) =
\exp(\bch( x_1,  x_2, \cdots, x_j))(\cdot)  \in \Aut (\widehat{\skein{\Sigma,J}})
\end{equation*} 
for any finite set $J \subset \partial \Sigma$.
\end{lemm}

\begin{proof}[Proof of Lemma \ref{lemm_zeta_inv_inj}]
For $x_1,x_2, \cdots, x_j ,
\in \shuugou{\epsilon x|\epsilon  \in \shuugou{\pm1},
x \in \mathcal{L}_{\mathrm{gem}} (\Sigma)}$,
we assume \begin{equation*}
\bch(x_1, \cdots, x_j) = 0. \end{equation*}
By Lemma \ref{lemm_theta}, we have
\begin{equation*}
\prod_{i=1}^j\theta (x_i) (\cdot) =\exp (\sigma (\bch (x_1, \cdots, x_j)))( \cdot) =\id (\cdot) 
\in \Aut (\widehat{\skein{\Sigma,J}})
\end{equation*}
for any finite set $J \subset \partial \Sigma$.
By Proposition \ref{prop_map_inj}, we have
$
\prod_{i=1}^j\theta (x_i) =\id.
$
This finishes the proof.
\end{proof}

\subsection{Well-definedness of $\zeta$}
\label{subsection_well_defined_zeta}
In this section, we prove $\theta$ is injective.
In order to check Putman's relation \cite{Pu2008},
we need the following.

\begin{lemm}
\label{lemm_relation_key}
Let $\Sigma$ be a compact surface with non-empty boundary.
We choose an element $x \in \widehat{\skein{\Sigma}}$ satisfying
$\sigma (x)(\widehat{\skein{\Sigma,J}}) =
\shuugou{0}$ for any finite set $J \subset \partial \Sigma$.
Then we have followings.
\begin{enumerate}
\item For any embedding $e' :\Sigma \to \tilde{\Sigma}$ and any element $\xi \in 
\mathcal{M}(\tilde{\Sigma})$, we have $\xi (e'(x)) =e'(x) \in 
\widehat{\skein{\tilde{\Sigma}}}$ where 
we also denote by $e'$ the homomorphism $\widehat{\skein{\Sigma}}
\to \widehat{\skein{\tilde{\Sigma}}}$ induced by $e'$.
\item For any embedding $e'':\Sigma \times I  \to I^3 \to \Sigma \times I$,
we have $x =e''(x) \in \Q [[A+1]] [\emptyset] \subset \widehat{\skein{\Sigma}}$
where we also denote by $e''$ the homomorphism $\widehat{\skein{\Sigma}}
\to \widehat{\skein{\Sigma}}$ induced by $e''$.
\end{enumerate}
\end{lemm}

\begin{proof}[Proof of (1)]
Since $\mathcal{M} (\tilde{\Sigma})$ is generated by Dehn twists,
it is enough to check $t_c (x) =x $ for  any simple closed curve $c$.
By assumption of $x$ and Theorem \ref{thm_Dehn_twist}, we have
\begin{align*}
t_c(x) =\exp(\sigma(L(c))(x) =x.
\end{align*}
This finishes the proof.
\end{proof}

\begin{proof}[Proof of (2)]
It is enough to show there exists an embedding
$e'':\Sigma \times I  \to I^3 \to \Sigma \times I$
such that $x = e''(x)$.
We choose compact connected surfaces $\tilde{\Sigma}$
 with non-empty connected boundary and
an embedding $e''': \tilde{\Sigma} \times I \to \Sigma \times I$
such that there
exists submanifolds $\Sigma'$, $\Sigma'' \subset \tilde{\Sigma}$
satisfying the followings.

\begin{itemize}
\item The intersection $\Sigma' \cap \Sigma''$ is empty.
\item There exists diffeomorphisms $\chi' : \Sigma \to \Sigma'$, 
$\chi'' : \Sigma \to \Sigma''$.
\item There exists en embedding $e_{I^3}: I^3 \to \Sigma \times I$ satisfying
\begin{equation*}
e'''\circ (\chi'' \times \id_I) (\Sigma \times I) \subset e''(I^3).
\end{equation*}
\item The embedding $e''' \circ (\chi' \times \id_I)$ induces the identity map of
$\widehat{\skein{\Sigma}}$.
\end{itemize}

\begin{figure}
\begin{picture}(300,540)
\put(0,240){\includegraphics[width=310pt]{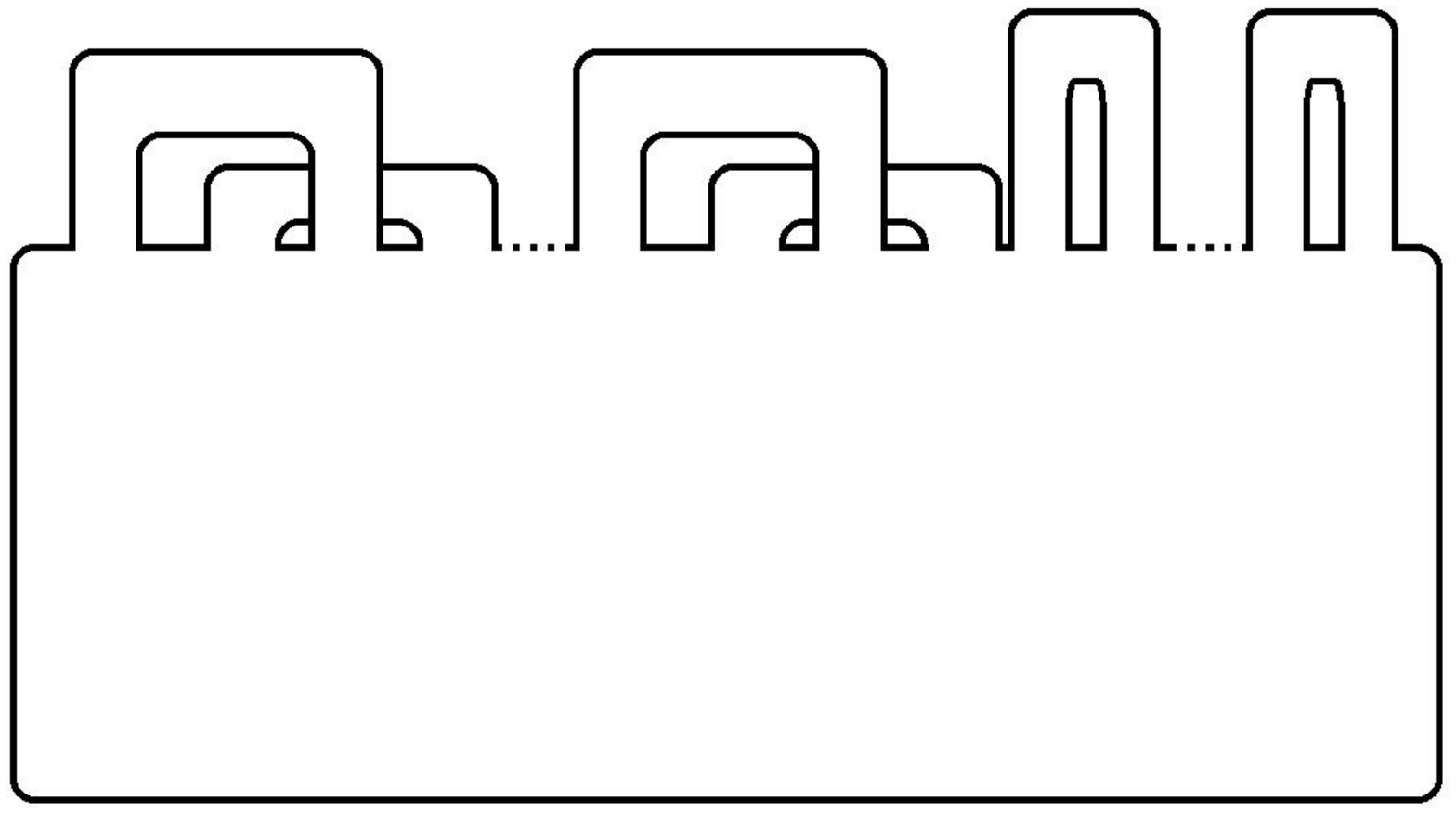}}
\put(0,60){\includegraphics[width=310pt]{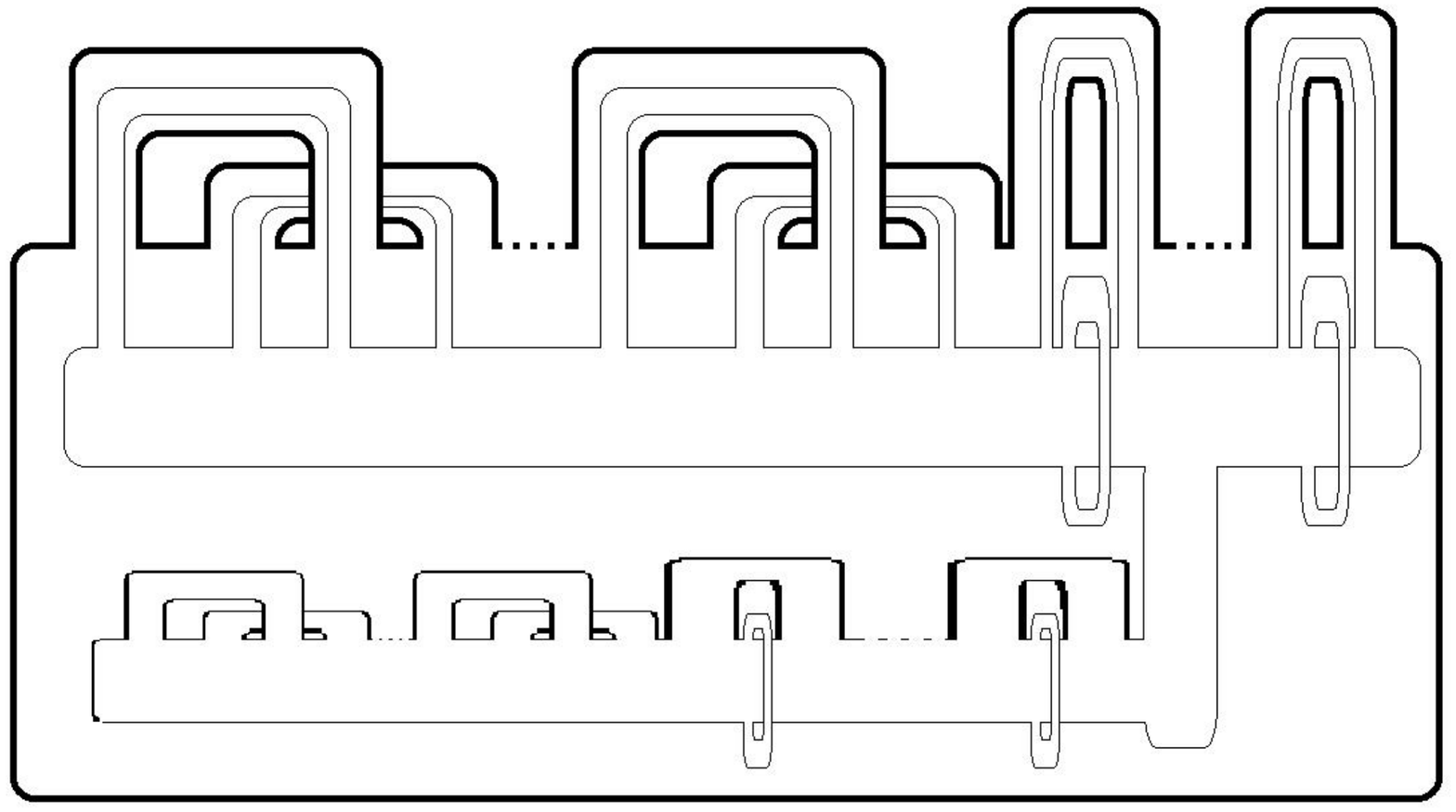}}
\put(0,-110){\includegraphics[width=310pt]{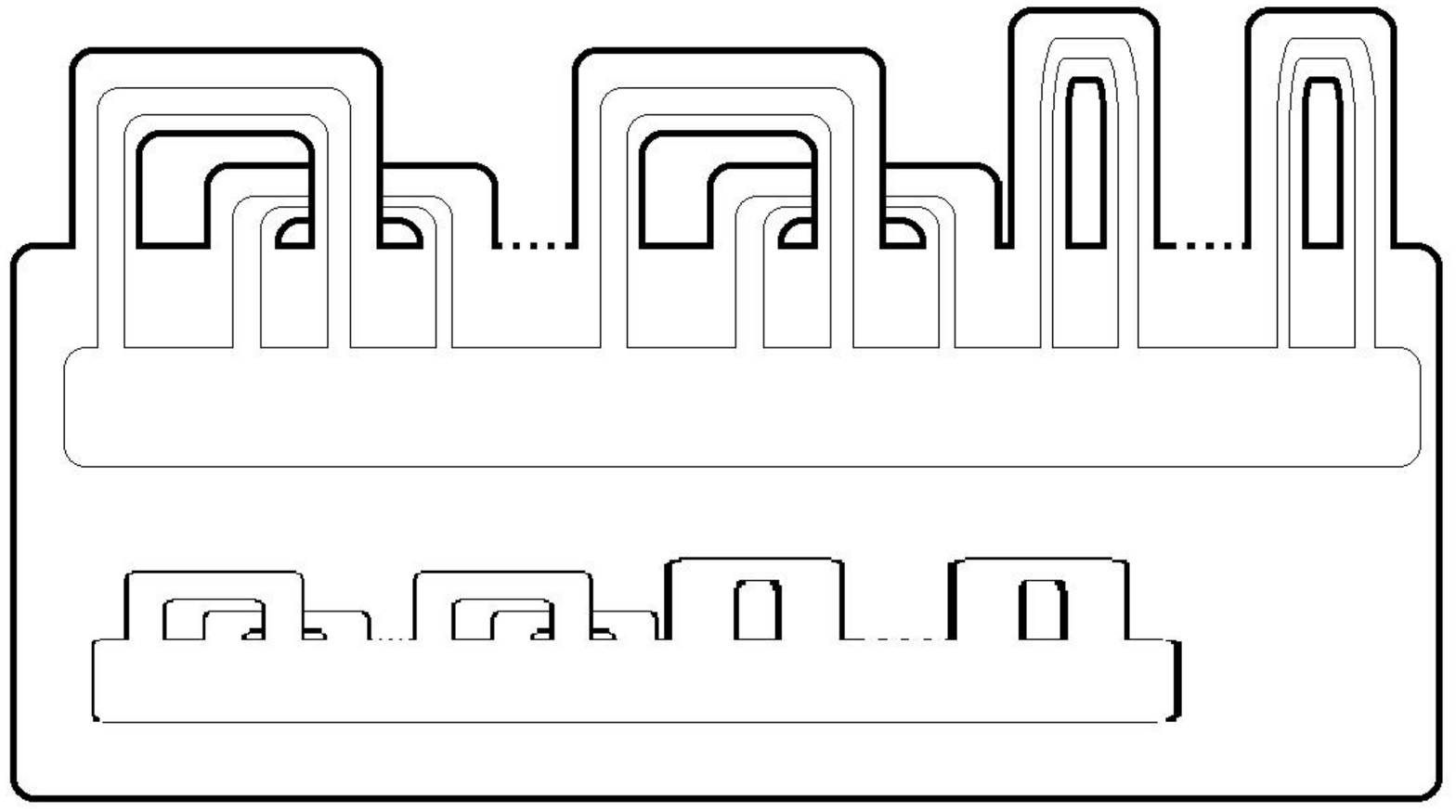}}
\put(90,370){{$\Sigma$}}
%\put(0,290){$e'''$}
\put(120,205){$\tilde{\Sigma}$}
\put(240,90){{$\Sigma'$}}
\put(90,35){$\Sigma''$}
\end{picture}
\caption{$e''': \tilde{\Sigma}\times I \to \Sigma \times I$}
\label{figure_embedding_Sigma_tilde}
\end{figure}

For example, see Figure \ref{figure_embedding_Sigma_tilde}.
By (1), we have $\chi'(x) =\chi''(x)$.
Hence, we obtain $x =e''' \circ (\chi' \times \id_I) (x) =e''' \circ 
(\chi'' \times \id_I)(x)$. The embedding $e''' \circ 
(\chi'' \times \id_I)$ satisfies the claim $e''$.
This finishes the proof.

\end{proof}

To check the relations (F.1), (F.2), ... (F.7) and (F.8)
in \cite{Pu2008}, we need the following.
We obtain the lemma by a straightforward calculation
and definition of the Baker-Campbell-Hausdorff series.
Lemma \ref{lemm_relation_F} (k) corresponds 
(F.k) for $k=1,2, \cdots ,8$.

\begin{lemm}
\label{lemm_relation_F}
\begin{enumerate}
\item For a pair $\shuugou{c_1,c_2}$
 of non-isotopic disjoint homologous curves, 
we have
\begin{equation*}
L(c_1)-L(c_2) =-(L(c_2)-L(c_1)).
\end{equation*}
\item For a pair $\shuugou{c_1,c_2}$
 of curves whose algebraic intersection number is $0$, we have
\begin{equation*}
\bch (\bch(L(c_1),L(c_2),-L(c_1),-L(c_2)),\bch(L(c_2),L(c_1),-L(c_2),-L(c_1)))=0.
\end{equation*}
\item For a triple $\shuugou{c_1,c_2,c_3}$
of non-isotopic disjoint homologous curves, we have
\begin{equation*}
\bch(L(c_1)-L(c_2), L(c_2)-L(c_3)) = L(c_1)-L(c_3).
\end{equation*}
\item If $\shuugou{c_1,c_2}$ is a pair of
non-isotopic disjoint homologous curves such that $c_1$ and $c_2$
are separating curves, we have
\begin{equation*}
L(c_1)-L(c_2) = \bch (L(c_1),-L(c_2)).
\end{equation*}
\item If $\shuugou{c_1,c_2}$ is a pair 
of non-isotopic disjoint homologous curves and
$\shuugou{c_3,c_2}$ is a pair
of curves whose algebraic intersection number is $0$ such that $c_1$ and $c_3$ are disjoint,
we have
\begin{equation*}
L(t_{c_3}(c_2)) -L(c_{1}) =\bch (\bch(L(c_3),L(c_2),-L(c_3),-L(c_2)),L(c_2)-L(c_1)).
\end{equation*}
\item  For
a pair $\shuugou{c_1,c_2}$ of curves whose algebraic intersection number is $0$
and $x \in \mathcal{L}_{\mathrm{gen}}(\Sigma)$, we have
\begin{equation*}
\bch(x,C(c_1,c_2),-x) =C(\theta(x)(c_1),\theta(x)(c_2)).
\end{equation*}
\item For a pair $\shuugou{c_1,c_2}$
 of non-isotopic disjoint homologous curves
and $x \in \mathcal{L}_{\mathrm{gen}}(\Sigma)$, we have
\begin{equation*}
\bch(x,L(c_1)-L(c_2),-x) =L(\theta(x)(c_1))-L(\theta(x)(c_2)).
\end{equation*}
\item For
a separating simple closed curve $c$
and $x \in \mathcal{L}_{\mathrm{gen}}(\Sigma)$, we have
\begin{equation*}
\bch(x,L(c),-x) =L(\theta(x)(c))
\end{equation*}
\end{enumerate}
\end{lemm}

In order to check lantern relation
in \cite{Pu2008}, we need the following.

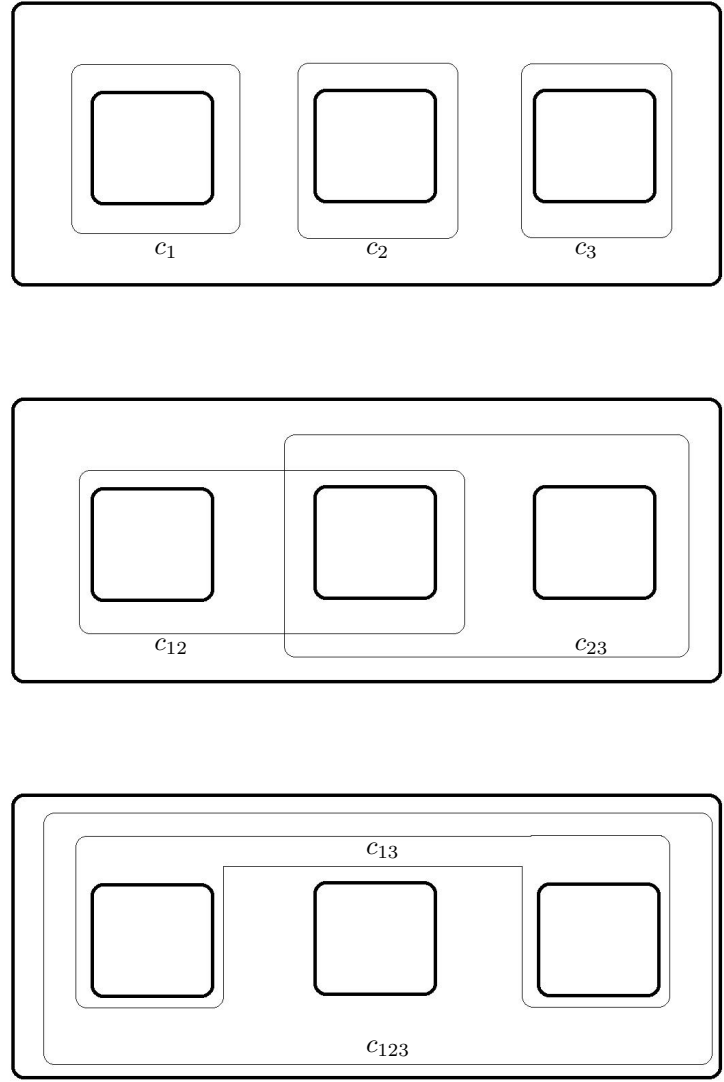
\begin{figure}
\begin{picture}(300,450)
\put(0,300){\includegraphics[width=310pt]{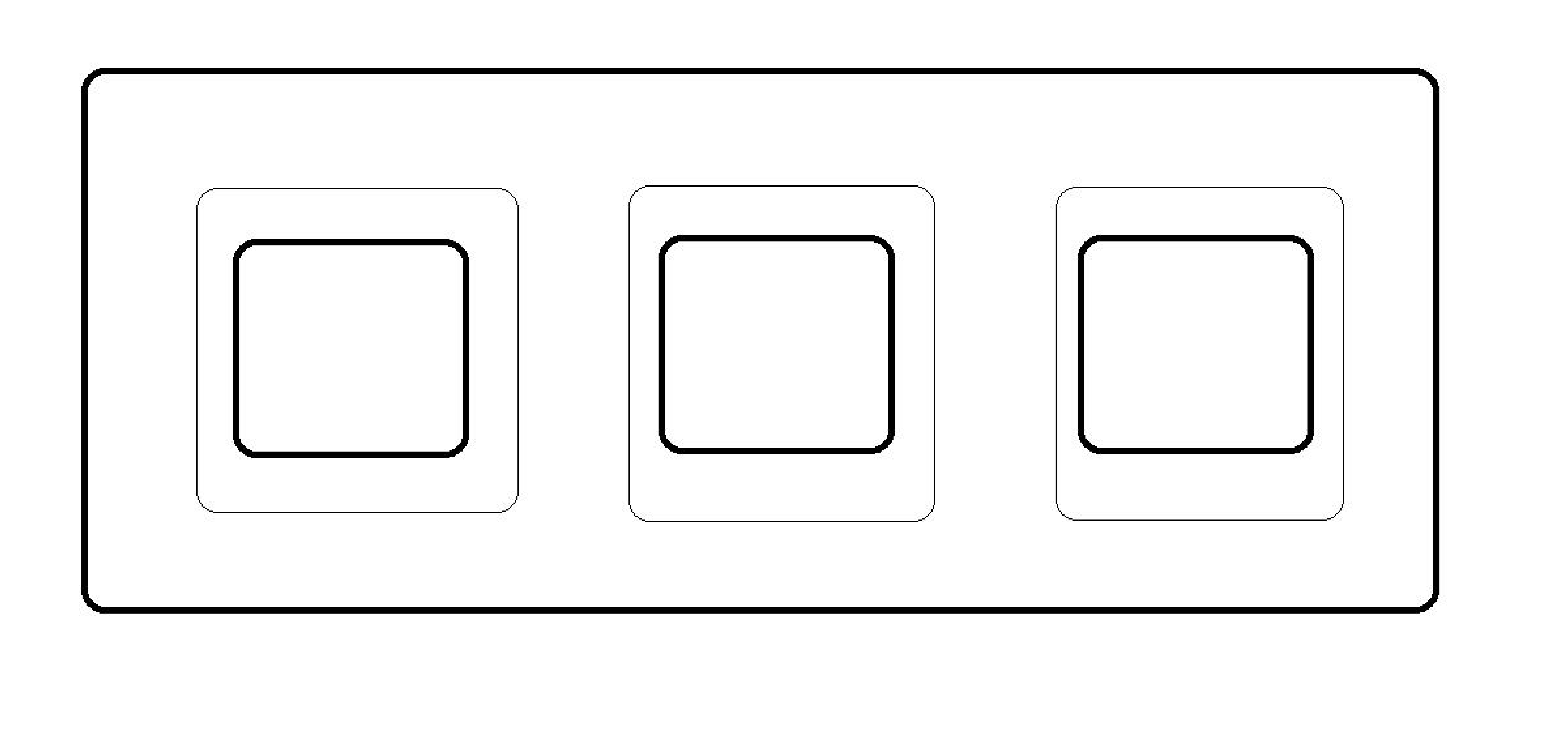}}
\put(0,150){\includegraphics[width=310pt]{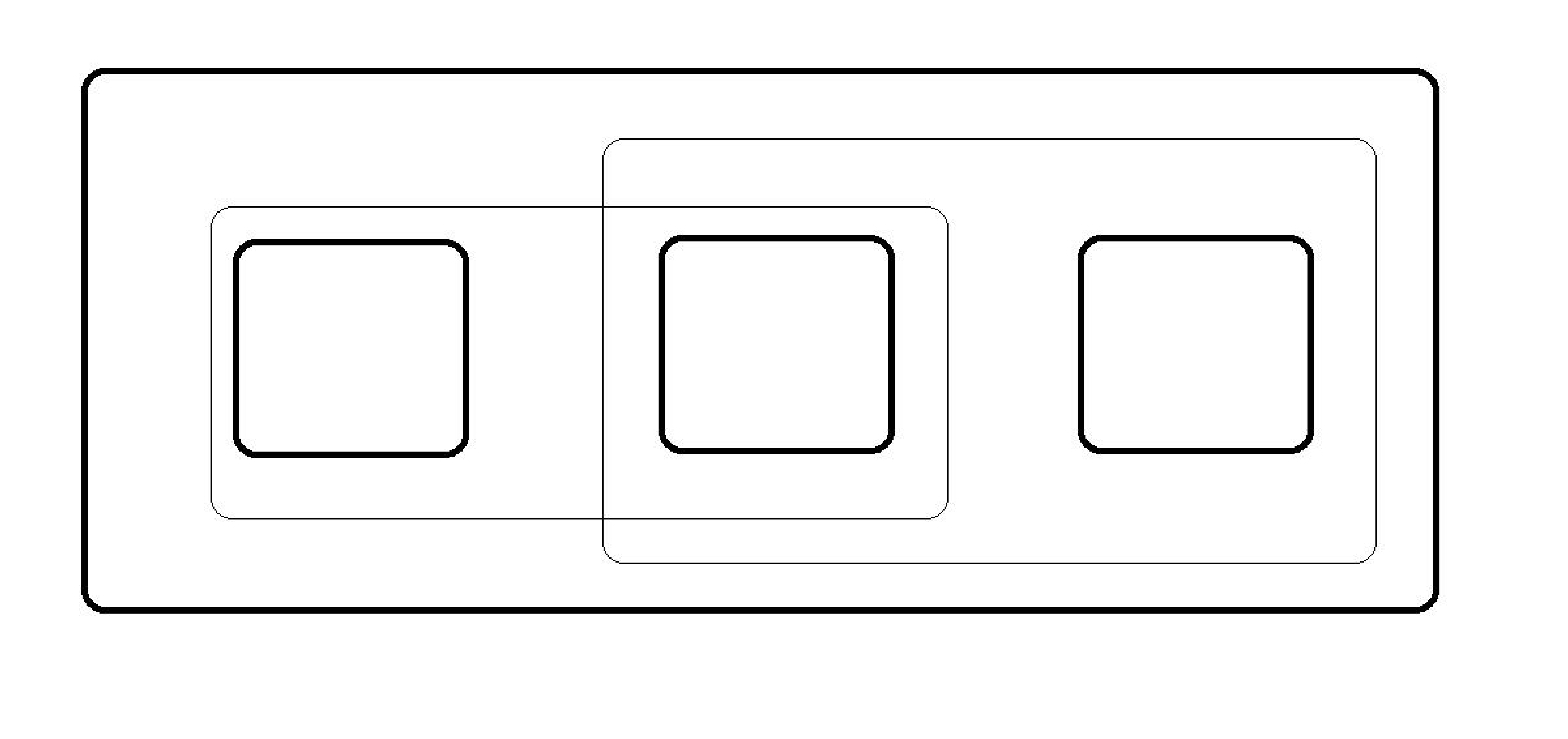}}
\put(0,0){\includegraphics[width=310pt]{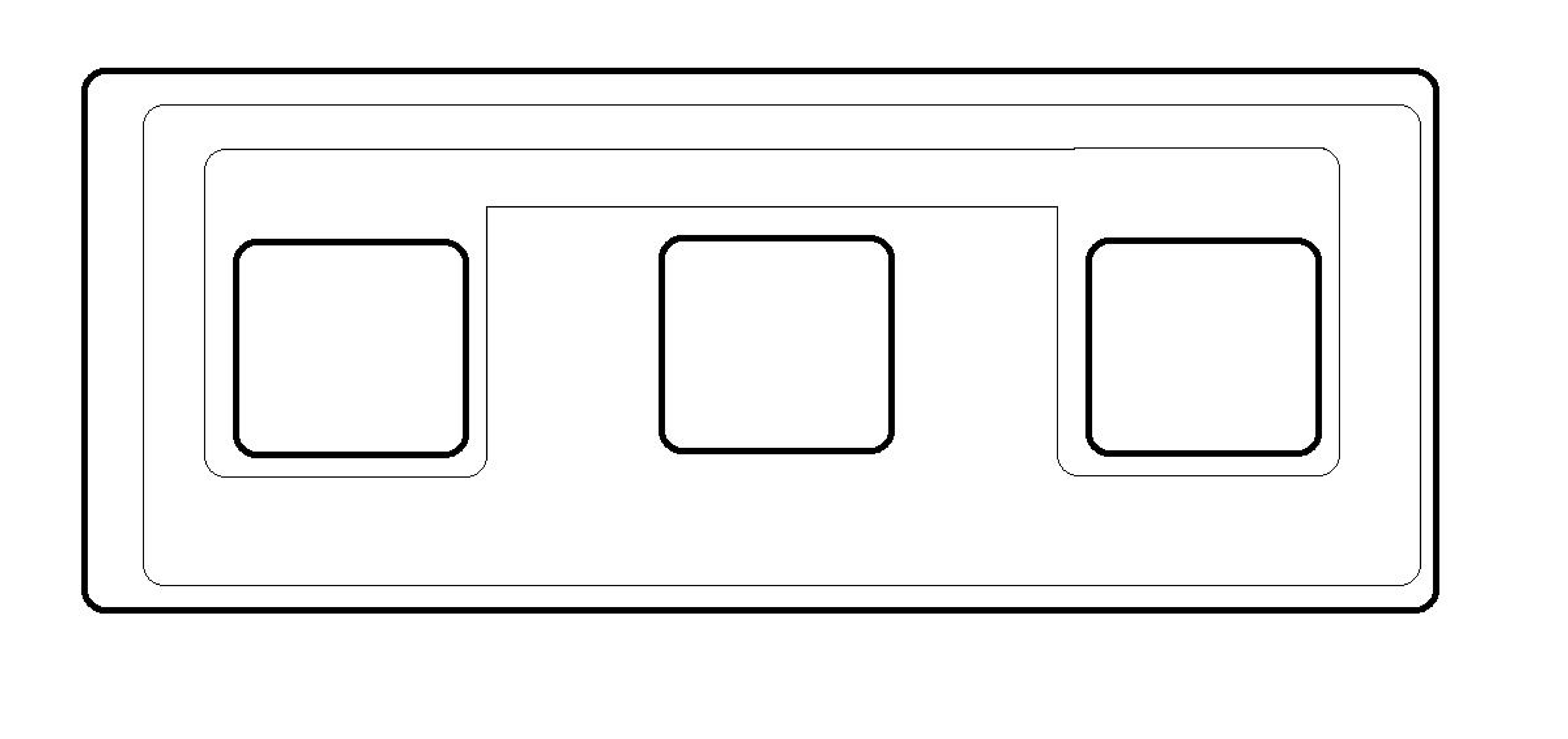}}
\put(70,337){$c_1$}
\put(150,337){$c_2$}
\put(229,337){$c_3$}
\put(70,187){$c_{12}$}
\put(229,187){$c_{23}$}
\put(150,35){$c_{123}$}
\put(150,110){$c_{13}$}
\end{picture}
\caption{The simple closed curves in $\Sigma_{04}$}
\label{figure_Sigma_0_4}
\end{figure}

\begin{lemm}
Let $\Sigma_{0,4}$ be a compact connected oriented surface
of genus $0$ and $4$ boundary components and
$c_1,c_2,c_3,c_{12},c_{23},c_{13}$ and $c_{123}$
the simple closed curves as Figure \ref{figure_Sigma_0_4}.
We have $\bch(L(c_{123},-L(c_{12}),-L(c_{23}),-L(c_{13}),L(c_1),L(c_2),L(c_3)))=0\in
\widehat{\skein{\Sigma_{0,4}}}$.
\end{lemm}

\begin{proof}
Let $e$ be an embedding $\Sigma_{0,4} \times I \to I^3$ as 
Figure \ref{figure_Sigma_0_4}.
Since 
\begin{align*}
&\exp (\bch(L(c_{123},-L(c_{12}),-L(c_{23}),-L(c_{13}),L(c_1),L(c_2),L(c_3))))(\cdot)
= \\
&t_{c_{123}}{t_{c_{12}}}^{-1}{t_{c_{23}}}^{-1}{t_{c_{31}}}^{-1}t_{c_{1}}t_{c_{2}}t_{c_3}(\cdot) = \id(\cdot)
\in \Aut (\widehat{\skein{\Sigma_{0.4},J}})
\end{align*}
for all finite set $J \subset \partial \Sigma$, we have
\begin{equation*}
\sigma(\bch(L(c_{123},-L(c_{12}),-L(c_{23}),-L(c_{13}),L(c_1),L(c_2),L(c_3))))
(\widehat{\skein{\Sigma_{0,4},J}}) = \shuugou{0}.
\end{equation*}
By Lemma \ref{lemm_relation_key},
it is enough to show $e(\bch(L(c_{123},-L(c_{12}),-L(c_{23}),-L(c_{13}),L(c_1),L(c_2),L(c_3))))
=0$. We remark $e([\widehat{\skein{\Sigma_{0,4}}}, \widehat{\skein{\Sigma_{0,4}}}])
=\shuugou{0}$. We have
\begin{align*}
&e(\bch(L(c_{123},-L(c_{12}),-L(c_{23}),-L(c_{13}),L(c_1),L(c_2),L(c_3)))) \\
&=(1-3+3)(\frac{-A+\gyaku{A}}{4\log(-A)} (\arccosh (\frac{A^2+A^{-2}}{2}))^2-
(-A+\gyaku{A})\log (-A))=0.
\end{align*}
This finishes the proof.

\end{proof}

We check the crossed lantern relation in \cite{Pu2008}.
Let $\Sigma_{1,2}$ be a connected compact surface
of genus $g=1$ with two boundary components. 
We denote simple closed curves in $\Sigma_{1,2}$
as in Figure \ref{figure_CL}.
We have $H_1(\Sigma_{1,2},\Q) =\Q a\oplus \Q b  \oplus \Q v$
where $a \defeq [c'_a], b \defeq [c_b], v \defeq [c_v] \in H_1 (\Sigma_{1,2},\Q)$.
In order to check the crossed lantern relation, 
we need the following.

\begin{figure}
\begin{picture}(300,140)
\put(0,-160){\includegraphics[width=340pt]{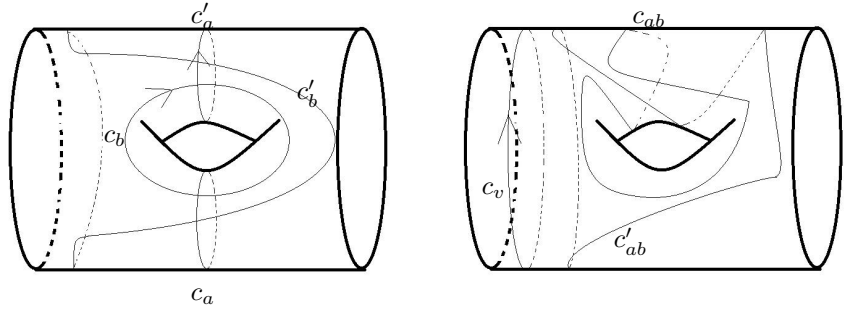}}
\put(47,70){$c_b$}
\put(120,86){$c'_b$}
\put(80,115){$c'_a$}
\put(80,10){$c_a$}
\put(190,50){$c_v$}
\put(240,30){$c'_{ab}$}
\put(247,115){$c_{ab}$}
\end{picture}
\caption{$\Sigma_{1,2}$}

\label{figure_CL}
\end{figure}

\begin{lemm}
\begin{enumerate}
\item The set $ \shuugou{\pm(L(c_a)-L(c'_a))
,\pm (L(c_b)-L(c'_b)), \pm (L(c_{ab})-L(c'_{ab}))}$
satisfies the conditions (\ref{equation_jouken_bch})
and (\ref{equation_bch_jouken_aut}).
\item We have $\bch (L(c_b)-L(c'_b),-L(c_a)+L(c'_a),-L(c_{ab})+L(c'_{ab}))=0$.
\end{enumerate}
\end{lemm}

\begin{proof}[Proof of (1)]
We have 
\begin{align*}
L(c_a)-L(c'_a) &=\frac{1}{4}\rho(-(a-v)\cdot(a-v)+a\cdot a ) \mod F^3 
\widehat{\skein{\Sigma_{0,4}}} \\
 &=\frac{1}{4}\rho(-v\cdot v+2v \cdot a ) \mod F^3 
\widehat{\skein{\Sigma_{0,4}}}, \\
L(c_b)-L(c'_b)&=\frac{1}{4}\rho(-b\cdot b+(b-v)\cdot (b-v)) \mod F^3 
\widehat{\skein{\Sigma_{0,4}}} \\
&=\frac{1}{4}\rho(v\cdot v-2v \cdot b ) \mod F^3 
\widehat{\skein{\Sigma_{0,4}}}, \\
L(c_{ab})-L(c'_{ab}) &=\frac{1}{4}\rho(-(a+b) \cdot (a+b) +(a+b-v)\cdot(a+b-v)) \mod F^3 
\widehat{\skein{\Sigma_{0,4}}} \\
&=\frac{1}{4}\rho(v\cdot v-2v \cdot (a+b) ) \mod F^3 
\widehat{\skein{\Sigma_{0,4}}}. \\
\end{align*}
By Lemma  \ref{lemm_bch_jouken}, this finishes the proof.
\end{proof}

\begin{proof}[Proof of (2)]
Let $e$ be an embedding $\Sigma_{0,4} \times I \to I^3$ as 
Figure \ref{figure_CL}.
Since 
\begin{align*}
&\exp (\sigma(\bch (L(c_b)-L(c'_b),-L(c_a)+L(c'_a),-L(c_{ab})+L(c'_{ab}))))(\cdot) \\
&=t_{c_b c'_b} \circ t_{c'_a c_a}\circ {t_{c_{ab} c'_{ab}}}^{-1}(\cdot) = \id(\cdot)
\in \Aut (\widehat{\skein{\Sigma_{1,2},J}})
\end{align*}
for all finite set $J \subset \partial \Sigma$, we have
\begin{equation*}
\sigma(\bch (L(c_b)-L(c'_b),-L(c_a)+L(c'_a),-L(c_{ab})+L(c'_{ab})))
(\widehat{\skein{\Sigma_{1,2},J}}) = \shuugou{0}.
\end{equation*}
By Lemma \ref{lemm_relation_key},
it is enough to show $e(\bch (L(c_b)-L(c'_b),-L(c_a)+L(c'_a),-L(c_{ab})+L(c'_{ab})))=0$.
Since $e(x (L(c_i)-L(c'_i))) =e((L(c_i)-L(c_i)) x) =0 $ for $i \in \shuugou{a, b,ab}$ and
any $x \in \widehat{\skein{\Sigma_{1,2}}}$, we have 
$e(\bch (L(c_b)-L(c'_b),-L(c_a)+L(c'_a),-L(c_{ab})+L(c'_{ab})))=0$.
This finishes the proof.
\end{proof}

In order to check Witt-Hall relation and commutator shuffle relation in 
\cite{Pu2008},
it is enough to show the following lemma.

\begin{lemm}
Let $\Sigma'$ be a compact surface, $D$ a open disk in $\Sigma'$
and $\Sigma''$ the surface $\Sigma' \backslash D$.
We fix the points $*_1$ and $*_2$ and the paths $\gamma_1$ and $\gamma_2$
as in Figure \ref{figure_WH_CS}.
We denote by $\mathcal{S}_{\mathrm{push}}(\Sigma'', \partial D)$
the set of all pair $\shuugou{c_1,c_2}$ of simple closed curves
satisfying the followings.
\begin{itemize}
\item
There exists a path $\gamma$ from $*_1$ to $*_2$ such that
$c_1 =\gamma_1 \cup \gamma$ and $c_2 =\gamma_2 \cup \gamma$.
\item
We have $[c_1] =[c_2] \in H_1 (\Sigma').$
\end{itemize}
For $\shuugou{c_{11},c_{21}} \cdots \shuugou{c_{1k},c_{2k}} \in 
\mathcal{S}_{\mathrm{push}}(\Sigma'', \partial D)$
and $\epsilon_1 \cdots \epsilon_k \in \shuugou{\pm1}$,
if $ \prod_{i=1}^k (t_{c_{1i}c_{2i}})^{\epsilon_i} = \id \in \mathcal{I}(\Sigma'')$,
then we have
\begin{equation*}
\bch(\epsilon_1 (L(c_{11})-L(c_{21})), \cdots, \epsilon_k(L(c_{1k})-L(c_{2k}))) =0.
\end{equation*}
\end{lemm}

\begin{figure}
\begin{picture}(300,140)
\put(0,-160){\includegraphics[width=340pt]{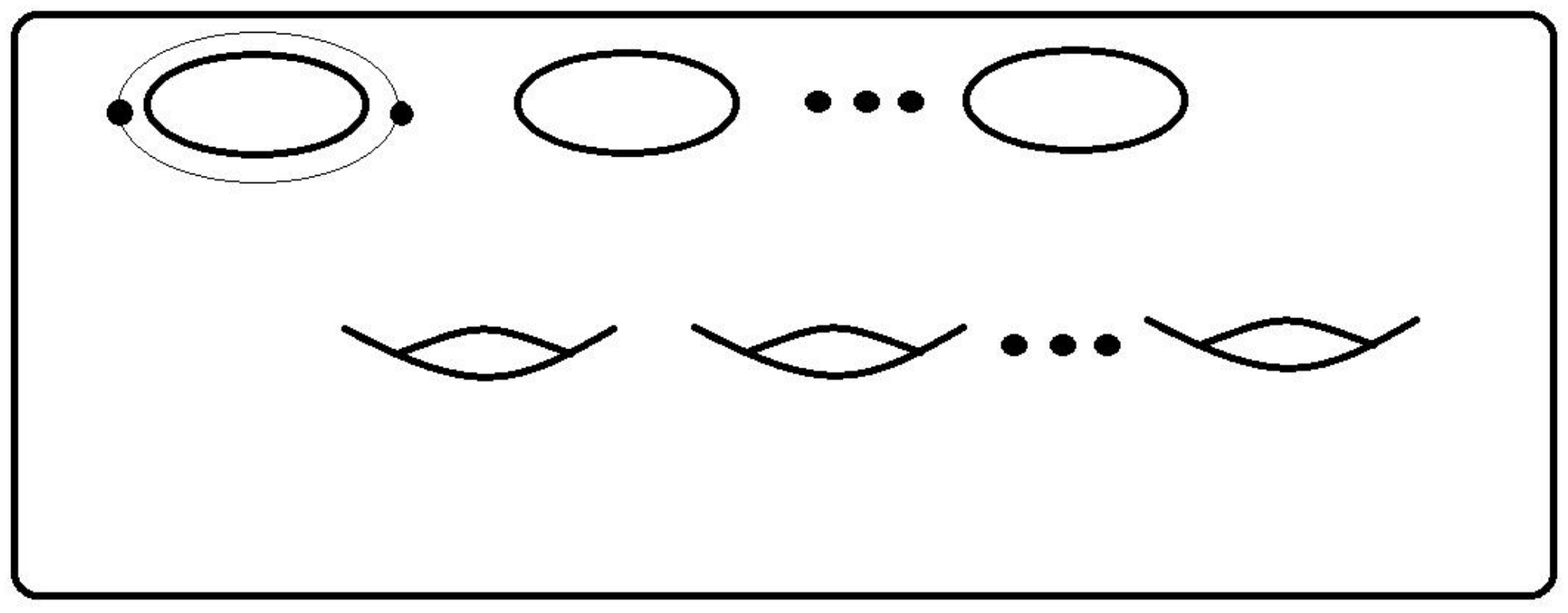}}
\put(60,90){$D$}
\put(100,86){$*_1$}
\put(20,86){$*_2$}
\put(90,100){$\gamma_1$}
\put(55,70){$\gamma_2$}
\end{picture}
\caption{$\Sigma'$}
\label{figure_WH_CS}
\end{figure}

\begin{proof}
Since 
\begin{equation*}
\exp (\sigma(\bch(\epsilon_1 (L(c_{11})-L(c_{21})), \cdots, \epsilon_k(L(c_{1k})-L(c_{2k})))
))(\cdot)
=\prod_{i=1}^k (t_{c_{1i}c_{2i}})^{\epsilon_i}(\cdot) = \id(\cdot)
\in \Aut (\widehat{\skein{\Sigma,J}})
\end{equation*}
for all finite set $J \subset \partial \Sigma$, we have
\begin{equation*}
\sigma(\bch(\epsilon_1 (L(c_{11})-L(c_{21})), \cdots, \epsilon_k(L(c_{1k})-L(c_{2k}))))
(\widehat{\skein{\Sigma,J}}) = \shuugou{0}.
\end{equation*}
We denote by $e'$ the embedding $\Sigma'' \to \Sigma'$.
We choose an embedding $e:\Sigma' \times I \to I^3$.
By Lemma \ref{lemm_relation_key},
it is enough to show  $e\circ e'(\bch(\epsilon_1 (L(c_{11})-L(c_{21})), \cdots,
 \epsilon_k((c_{1k})-L(c_{2k}))))=0$.
Since $e'(L(c_{1i})-L(c_{2i}))$ for $i \in \shuugou{1,2, \cdots,k}$,
we have $e'(\bch(\epsilon_1 (L(c_{11})-L(c_{21})), \cdots,
 \epsilon_k(L(c_{1k})-L(c_{2k}))))=0$.
This finishes the proof.
\end{proof}

By above lemmas, we have the main theorem.

\begin{thm}
The group homomorphism $\theta :
I \skein{\Sigma} \to  \mathcal{I}(\Sigma)$ is
an isomorphism. In other words,
$\zeta \defeq \theta^{-1} :
\mathcal{I} (\Sigma) \to \widehat{\skein{\Sigma}}$ is embedding.

\end{thm}

Since $\mathcal{I} (\Sigma) $ is generated by
$\shuugou{t_{c_1c_2}|c_1,c_2: \mathrm{BP}}$, 
we have the following.

\begin{cor}
We have $I \skein{\Sigma} =\zettaiti{\mathcal{L}_{bp}(\Sigma)}$.
Furthermore $\zeta$ can be defined by
$\zeta(t_{c_1c_2})=L(c_1)-L(c_2)$.
\end{cor} 

\subsection{The 1st Johnson homomorphism and $\zeta$}
\label{subsection_johnson_zeta}
In this subsection, we prove the following.

\begin{thm}
\label{thm_johnson_zeta}
The isomorphism
$\zeta :\mathcal{I} (\Sigma) \to I \skein{\Sigma}$ 
induces $\tau_\zeta :\mathcal{I} (\Sigma) \to
F^3 \widehat{\skein{\Sigma}}/
F^4 \widehat{\skein{\Sigma}} \stackrel{\lambda^{-1}}{\simeq}
H_1 \wedge H_1 \wedge H_1$. Then we have $\tau_\zeta =\tau$
where $\tau$ is the 1st Johnson homomorphism.
\end{thm}

\begin{proof}
By \cite{Johnson80} Lemma 4B, it is enough to show
$\tau_\zeta (t_{c_1c_2}) =\sum_{i=1}^k 
[a_i] \wedge [b_i] \wedge [b_{k+1}]$ where
$c_1$ and $c_2$ are disjoint simple closed curves
which is presented by $(b_{k+1})_\square$ and
$( \prod_{i=1}^k[a_i,b_i] b_{k+1})_\square$, respectively.
By Lemma \ref{lemm_bounding_pair}, we have
$L(c_1)-L(c_2) =\lambda (\sum_{i=1}^k 
[a_i] \wedge [b_i] \wedge [b_{k+1}]) \mod F^4 \widehat{\skein{\Sigma}}$.
This finishes the proof.
\end{proof}

Since $\mathcal{K} (\Sigma) \defeq  \ker \tau$ is generated by
$\shuugou{t_c|c:\mathrm{sep. s.c.c.}}$, we have the following.

\begin{cor}
The subgroup $F^4 \widehat{\skein{\Sigma}} \cap I\skein{\Sigma}$
 is generated by $\mathcal{L}_{\mathrm{sep}}(\Sigma)$.
Furthermore, the restriction of $\zeta$ to $\mathcal{K} (\Sigma)$
$\zeta_{|\mathcal{K}(\Sigma)}:\mathcal{K}(\Sigma) \to
F^4 \widehat{\skein{\Sigma}} \cap I\skein{\Sigma}$
is an isomorphism defined by 
$\zeta (t_c) =L(c)$ for a separating simple closed curve $c$.
\end{cor}

\end{document}